\documentclass[12pt,a4paper]{amsart}
\usepackage[latin9]{inputenc}
\usepackage{graphicx, color, enumerate}
\usepackage{amsmath, amssymb}

\usepackage{amsthm} 
\usepackage{newtxmath,newtxtext}
 \usepackage{bm}
 \usepackage{slashbox}    
\usepackage{lineno}

\theoremstyle{plain} 
\newtheorem{theorem}{Theorem}[section]
\newtheorem{lemma}[theorem]{Lemma}
\newtheorem{corollary}[theorem]{Corollary}
\newtheorem{proposition}[theorem]{Proposition}
\newtheorem{fact}[theorem]{Fact}

\theoremstyle{definition}
\newtheorem{definition}[theorem]{Definition}
\newtheorem{remark}[theorem]{Remark}
\newtheorem{example}[theorem]{Example}

 \usepackage{caption}

\usepackage{color}
\usepackage{graphicx,color}
\usepackage{amsmath, amssymb, graphics}

\graphicspath{{./Figures/}}

\DeclareMathOperator{\Hess}{Hess}

\DeclareMathOperator{\arcsinh}{arcsinh}

\begin{document}
\title[Singularities of timelike minimal surfaces]{Singularities of generalized timelike minimal surfaces\\ in Lorentz-Minkowski 3-space}
\author[S. Akamine]{Shintaro Akamine}
\address[Shintaro Akamine]{
Department of Liberal Arts, College of Bioresource Sciences,
Nihon University, 
1866 Kameino, Fujisawa, Kanagawa, 252-0880, Japan}
\email{akamine.shintaro@nihon-u.ac.jp}

\date{\today}


\begin{abstract}
A timelike minimal surface in Minkowski 3-space is a surface whose induced metric is Lorentzian and with vanishing mean curvature. Such surfaces have many kinds of singularities. In this paper, we prove existence and non-existence theorems of singularities of timelike minimal surfaces, and show that various diffeomorphism types of singularities that do not appear on these Riemannian counter parts, such as minimal surfaces in Euclidean space and maximal surfaces in Minkowski space, appear on timelike minimal surfaces. We also give criteria for cuspidal butterfly, cuspidal $S_1$ singularity, $(2,5)$-cuspidal edge, cuspidal beaks and $D_4$ singularity of timelike minimal surfaces. Finally, duality and invariance theorems for these singularities and examples are given.  
\end{abstract}

\maketitle


\section{Introduction}\label{Sec.1}
A spacelike surface in the three-dimensional Lorentz-Minkowski space $\mathbb{L}^3$ whose mean curvature vanishes identically is called a {\it maximal surface}. It is known that only plane is a complete maximal surface \cite{Calabi,CY,Kobayashi}, and hence maximal surfaces with singular points has been considered as a natural object of study. Here, a {\it singular point} or {\it singularity} is a point at which the surface is not immersed.
There have been many remarkable results concerning singularities themselves. In particular, diffeomorphism types of various singularities have been determined by using a pair $(g, \omega)$ of a meromorphic function $g$ and a holomorphic $1$-form $\omega$ on a Riemann surface, which is called {\it Weierstrass data} and describes a maximal surface. See \cite{ FSUY, KKSY, KY,OT, OT2,UY} and their references. Moreover, Fujimori, Saji, Umehara and Yamada \cite{FSUY} proved that generic singularities on maximal surfaces (more precisely, the class of maximal surfaces with singularities called {\it maxfaces}) consist of cuspidal edges, swallowtails and cuspidal cross caps. For their standard forms, see figure \ref{Fig1}.

Minimal surfaces in the three-dimensional Euclidean space $\mathbb{E}^3$ are also classical  and important objects of study in differential geometry. On such minimal surfaces, isolated singular points, called {\it branch points}, where the tangent plane degenerates to a single point, often appear. 
To give some well-known historical examples of such surfaces, Catalan's minimal surface \cite{Catalan} and Henneberg's non-orientable minimal surface \cite{Henneberg} are minimal surfaces with branch points locally having the diffeomorphism type $f_{D_4^-}(u,v)=(uv,u^2- 3v^2,u^2v- v^3)$ called the $D_4^-$ singularity, see Figure \ref{Fig:Catalan} (and also Figure \ref{Fig1} for $D_4^-$ singularity). 

\begin{figure}[!h]
\begin{center}
\begin{tabular}{c}
\begin{minipage}{0.4\hsize}
\begin{center}
\vspace{-1.0cm}
\includegraphics[clip,scale=0.40,bb=0 0 500 400]{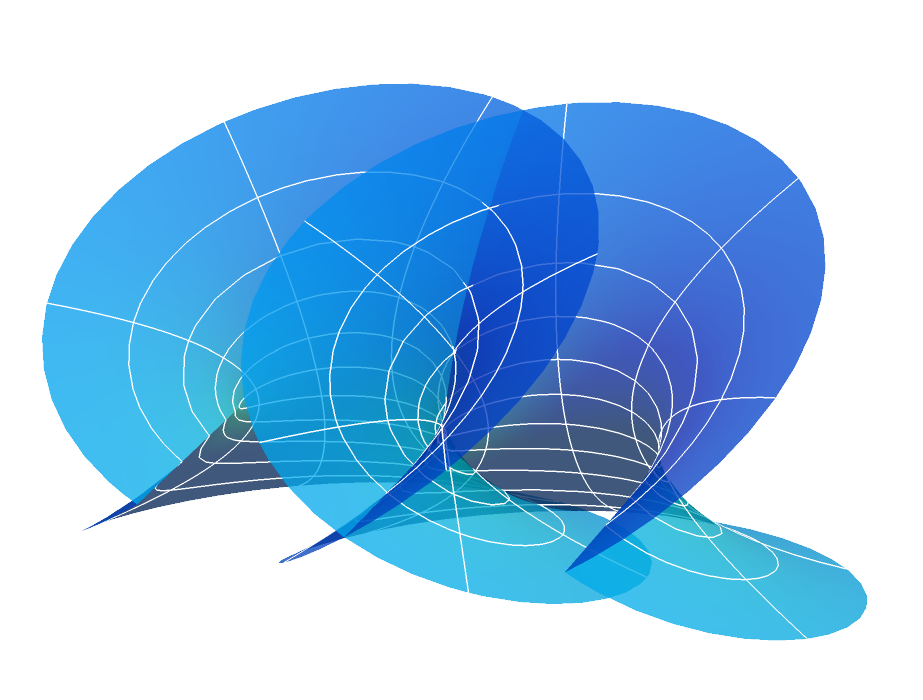}
\vspace{0.5cm}
\end{center}
\end{minipage}
\hspace{0.8cm}
\begin{minipage}{0.4\hsize}
\begin{center}
\vspace{-0.45cm}
\includegraphics[clip,scale=0.35,bb=0 0 500 400]{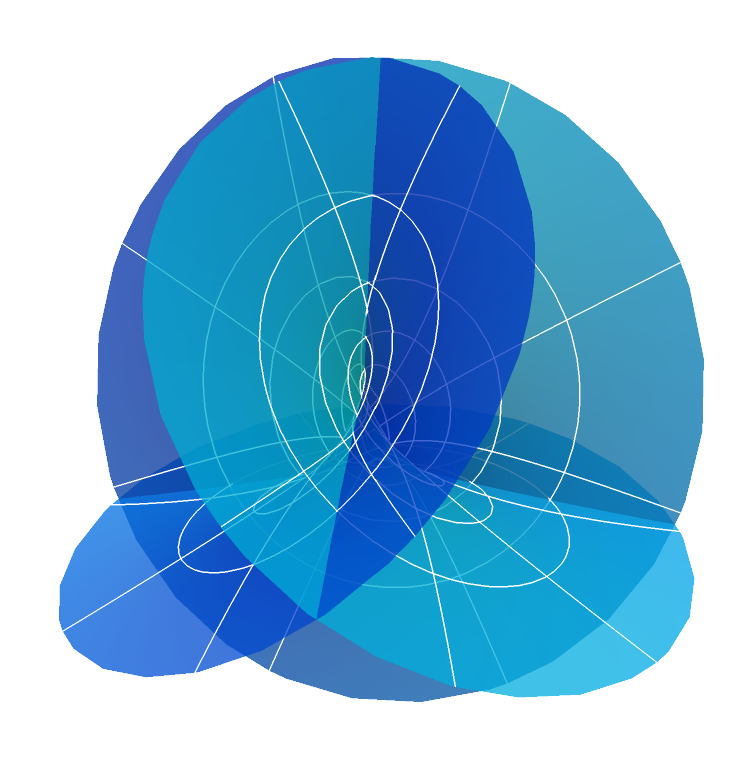}
\vspace{0.3cm}
\end{center}
\end{minipage}

\end{tabular}
\end{center}
\vspace{-1.3cm}
\caption{Catalan's minimal surface (left) and Henneberg's minimal surface (right) with $D_4^-$ singularities as branch points.}
 \label{Fig:Catalan}
\end{figure}

In this paper, as Lorentzian counter parts of the above surfaces, we deal with {\it timelike minimal surfaces}, which are surfaces in $\mathbb{L}^3$ whose induced metric is Lorentzian and whose mean curvature vanishes identically. The aim of this paper is to show that when a surface is equipped with Lorentzian metric, the singularities appearing on it are quite different from those appearing on the above Riemmannian counter parts, namely minimal surfaces in $\mathbb{E}^3$ and maximal surfaces in $\mathbb{L}^3$.

As conformal minimal surfaces are characterized by the property that their coordinate functions are harmonic (i.e. the surface $f$ satisfies the Laplace equation $\Delta f=\bm{0}$), each conformal timelike minimal surface $f$ is also characterized by the wave equation $\Box f =\bm{0}$, where $\square$ is the d'Alembert operator. For local null coordinates $(u,v)$ so that $f^*\langle, \rangle =2Fdudv$, the wave equation above is equivalent to $f_{uv}=\bm{0}$. Here, $\langle,\rangle$ is the canonical Lorentzian inner product of $\mathbb{L}^3$. Hence, arbitrary timelike minimal surface $f$ is locally decomposed into the sum of two null curves as
\[
f(u,v) = \frac{\varphi(u)+\psi(v)}{2},\quad \langle \varphi', \varphi'\rangle = \langle \psi', \psi'\rangle =0,
\] 
where $\varphi'=d\varphi/du$ and $\psi'=d\psi/dv$. For example, if we take two regular null curves $\varphi(u)=\left(-u-u^3/3, u-u^3/3,u^2\right)$ and $\psi(v)=\left(v+v^3/3, v-v^3/3, v^2\right)$, we obtain the timelike Enneper surface $f=(\varphi(u)+\psi(v))/2$. Singular points of this surface correspond to points $(u,v)$ satisfying $uv=-1$ at which non-zero velocity vectors $\varphi'$ and $\psi'$ are linearly dependent. Singular points of this surface consist of cuspidal edges and swallowtails, see \cite[Example 4.5]{Akamine} for more details. 
As in this example, a timelike minimal surface with singularities that can be locally decomposed into a sum of regular null curves is called a {\it minface}, and singularities on minfaces are investigated by takahashi \cite{T} (see also \cite{Akamine}) in the same way as maxfaces. 
In particular, the criteria for cuspidal edge, swallowtail and cuspidal cross cap are given in \cite{T}. These are generic singularities on maximal surfaces, but if we consider more complicated diffeomorphism types, we will show in Section \ref{Sec:Gsing} that singularities that do not appear on the maximal surfaces do appear on minfaces.

On the other hand, once we consider the situation where the generating null curves $\varphi$ and $\psi$ are no longer regular, more types of singularities start to appear on timelike minimal surfaces.
For example, let us consider the periodic closed null curve 
\[
\gamma(s)=\int\left(1, \cos{3s}, \sin{3s} \right)\cos{2s}ds
\]
and the surface $f(u,v)=(\gamma(u)+\gamma(v))/2$. Since $f$ is periodic in both $u$ and $v$ directions, $f$ gives a timelike minimal torus, see Figure \ref{Fig:D4Torus}. The tangent plane of the surface $f$ becomes zero-dimensional, i.e. $df_p=\bm{0}$ at a point $p$ where the null curves $\varphi$ and $\psi$ simultaneously fail to be regular, and $f$ is locally diffeomorphic to the standard $D_4^+$ singularity $f_{D_4^+}(u,v)=(uv,u^2+ 3v^2,u^2v+ v^3)$ at $p$. As shown in \cite{OT2}, it is known that such $D_4^+$ singularities never appear on minimal or maximal surfaces. It has also been pointed out in \cite{KN1,KN2} that timelike minimal surfaces with singularity where at least one of the generating null curves is not regular or smooth appear in string theory in physics.

\begin{figure}[htb]
\vspace{-2.0cm}

\begin{center}
 \begin{tabular}{{c@{\hspace{-20mm}}c@{\hspace{-20mm}}c}}
\hspace{-8mm}   \resizebox{8.0cm}{!}{\includegraphics[clip,scale=0.33,bb=0 0 555 449]{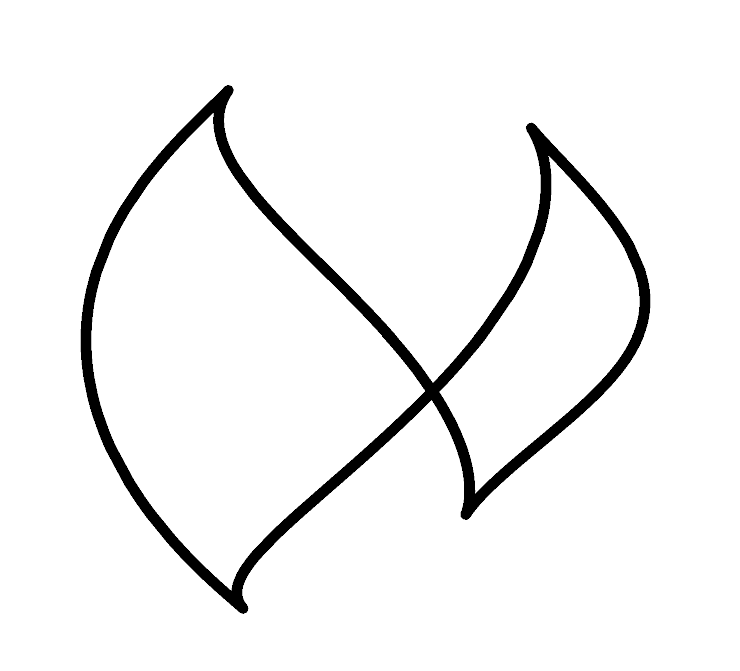}}&
 \hspace{-6mm}  \resizebox{6.7cm}{!}{\includegraphics[clip,scale=0.30,bb=0 0 555 449]{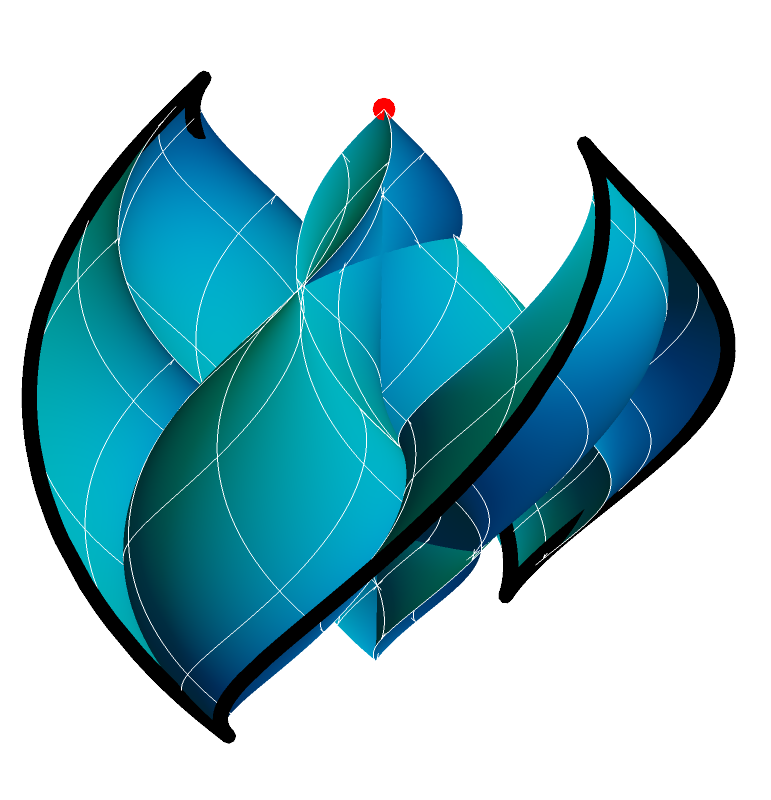}}&
 \hspace{+3mm} \resizebox{7.2cm}{!}{\includegraphics[clip,scale=0.38,bb=0 0 555 449]{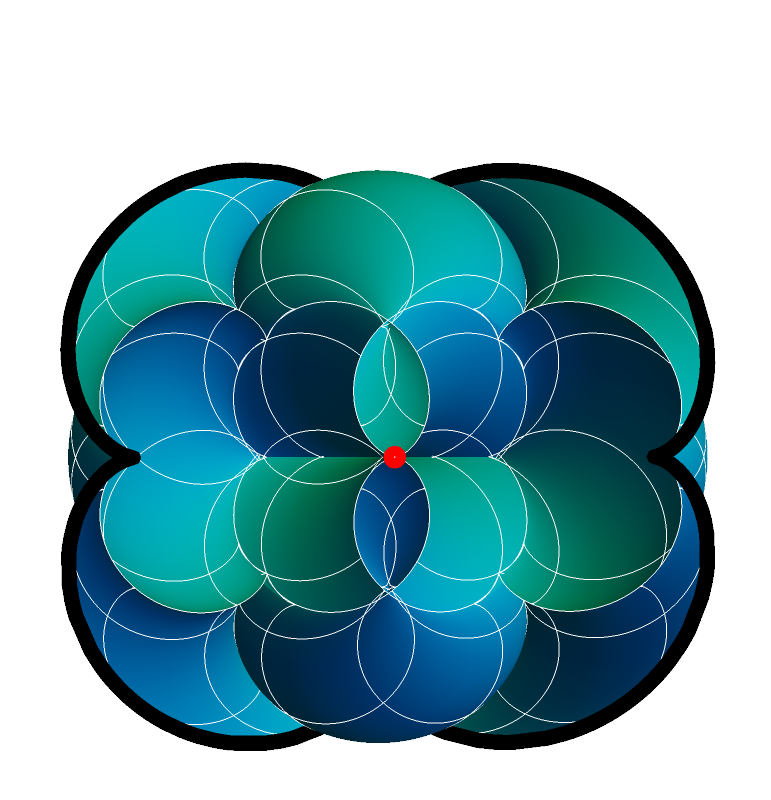}} \\
  {\hspace{-35mm}\footnotesize  A closed null curve.} &
  {\hspace{-28mm}\footnotesize  A timelike minimal torus.} &
  {\hspace{-20mm}\footnotesize  Another view.}
 \end{tabular}
 \caption{A periodic closed null curve $\gamma$ (left) and a timelike minimal torus with $D_4^+$ singularity (red dot) generated by $\gamma$ (center and left).}
 \label{Fig:D4Torus}
\end{center}
\end{figure}

In this paper, as seen in the example above, we consider a wider class of timelike minimal surfaces than minfaces and show that many diffeomorphism types of singularities appear on them that do not appear in the Riemannian counter parts. 

The main results of this paper are as follows.

\begin{itemize}
 \setlength{\leftskip}{-0.8cm}
  \setlength{\rightskip}{0.5cm}
\item Criteria for cuspidal butterfly and cuspidal $S_1^+$ singularity on timelike minimal surfaces are given, whereas the nonexistence of cuspidal $S_1^-$ singularity, which usually appears on maximal surfaces, is shown on timelike minimal surfaces (Theorem \ref{Thm:CBF_CS1}).
\vspace{0.2cm}

\item A duality between cuspidal butterfly and cuspidal $S_1^+$ singularity holds: a timelike minimal surface $f=(\varphi+\psi)/2$ has a cuspidal butterfly at a singular point $p$ if and only if its conjugate surface $\hat{f}:=(\varphi-\psi)/2$ has a cuspidal $S_1^+$ singularity at $p$ (Corollary \ref{cor:duality_CBF}). 
\vspace{0.2cm}

\item Criterion for $(2,5)$-cuspidal edge on timelike minimal surfaces is given and the existence of $(2,2k+1)$-cuspidal edge for arbitrary positive integer $k$ is also shown (Theorem \ref{criterion_(2,5)CE} and Proposition \ref{ExistenceCusp}). Furthermore, invariance of $(2,5)$-cuspidal edges under isometric and anti-isometric deformations is given (Corollary \ref{cor:duality_(2,5)_cusp}). 
\vspace{0.2cm}

\item Degenerate singularities (see Section \ref{Sec.frontals} for definition) are considered. In particular, criteria for cuspidal beaks and $D_4^+$ singularity are given, whereas the nonexistence of cuspidal lips and $D_4^-$ singularity is shown on timelike minimal surfaces (Proposition \ref{prop:BeaksLips} and Proposition \ref{D_4criterion}). Furthermore, invariance of $D_4^+$ singularities under isometric and anti-isometric deformations is given (Corollary \ref{cor:duality}). 
\end{itemize}


Finally, it is worth mentioning some related studies. While timelike minimal surfaces are locally represented as translation surfaces of null curves, singularities of translation surfaces in $\mathbb{E}^3$ are studied in \cite{FT}. Singularities of timelike surfaces in $\mathbb{L}^3$ with non-zero constant mean curvature are also studied in \cite{BS}. 


\section{Timelike minimal surfaces with singularities}
\subsection{Singularities on fronts and frontals}\label{Sec.frontals}
First we recall the singularity theory of frontals and fronts, see \cite{AGV, FSUY, SUY09, UY} for details. Let $U$ be a domain in $\mathbb{R}^2$ and $u$, $v$ are local coordinates on $U$. A smooth map $f\colon U \longrightarrow \mathbb{R}^3$ is called a {\it frontal} if there exists a smooth map $n$ from $U$ to the unit sphere $S^2$ such that $n$ is perpendicular to $df(TU)$ with respect to the Euclidean metric $\langle \  ,\  \rangle_E$ of $\mathbb{R}^3$. We call $n$ the {\it unit normal vector field} of $f$. Moreover if the {\it Legendrian lift} $L=(f,n)\colon U \longrightarrow \mathbb{R}^3\times S^2$ is an immersion, $f$ is called a {\it front}. A point $p\in U$ where $f$ is not an immersion is called a {\it singular point} of the frontal $f$, and we call the set of singular points of $f$ the {\it singular set}. A singular point $p$ is said to be of {\it rank one} (resp.~{\it rank zero}) if ${\rm rank}(Jf_p)=1$ (resp.~ ${\rm rank}(Jf_p)=0$). 


The function $\lambda=\det(f_u, f_v, n)$ on $(U;u,v)$ is called the {\it signed area density function} of the frontal $f$. A singular point $p$ is called {\it non-degenerate} (resp.~{\it degenerate}) if $d\lambda_p \neq0$ (resp.~$d\lambda_p=0$). 

The set of singular points of $f$ corresponds to zeros of $\lambda$. When a singular point $p$ of a frontal $f$ is non-degenerate, there exist an $\varepsilon>0$ and a regular curve $\gamma=\gamma(t): (-\varepsilon, \varepsilon)\longrightarrow U$ such that $\gamma(0)=p$ and the image of $\gamma$ coincides with the singular set of $f$ around $p$. We call $\gamma$ the {\it singular curve} and the direction of $\gamma'={d\gamma}/{dt}$ the {\it singular direction}. On the other hand, there exists a non-zero vector $\eta \in \mathrm{Ker}(df_p)$ because $p$ is non-degenerate. We call $\eta$ the {\it null direction}. 
 
The following functions on the singular curve $\gamma$ are used in the criteria for various singularities.
\begin{equation}\label{eq:delta_psi}
\delta(t)=\det{(\gamma'(t),\eta(t))},\quad \psi(t)=\det{(d{f}(\gamma'(t)),n(\gamma(t)),d{n}(\eta(t)))}.
\end{equation}

Let $U_i$ be domains of $\mathbb{R}^2$ and $p_i$ be points in $U_i$ ($i=1, 2$). Two smooth map germs $f_1\colon (U_1,p_1)\longrightarrow (\mathbb{R}^3,f_1(p_1))$ and $f_2 \colon (U_2,p_2)\longrightarrow (\mathbb{R}^3,f_2(p_2))$ are {\it $\mathcal{A}$-equivalent} if there exist diffeomorphism germs $\Phi\colon (\mathbb{R}^2,p_1) \longrightarrow (\mathbb{R}^2,p_2)$ and $\Psi\colon (\mathbb{R}^3,f_1(p_1)) \longrightarrow (\mathbb{R}^3,f_2(p_2))$ such that $f_2 = \Psi \circ f_1\circ \Phi^{-1}$. 

We set $f^{(2,2k+1)}_{CE}$, $f_{SW}$, $f_{CCR}$, $f_{CBF}$, $f_{CS_1^{\pm}}$, $f_{CBK}$, $f_{CL}$, $f_{D_4^{\pm}}\colon (\mathbb{R}^2,0)\longrightarrow (\mathbb{R}^3,0)$ as
\begin{align}
&f^{(2,2k+1)}_{CE}(u,v)=(u,v^2,v^{2k+1}),\quad f_{SW}(u,v)=(3u^4+u^2v,4u^3+2uv,v), \nonumber \\
&f_{CCR}(u,v)=(u,v^2,uv^3),\quad f_{CBF}(u,v)=(u,5v^4+2uv,4v^5+uv^2-u^2),  \nonumber \\
&f_{CS_1^{\pm}}(u,v)=(u,v^2,v^3(u^2\pm v^2)),\quad f_{CBK}(u,v)=(v,u^3-uv^2,3u^4-2u^2v^2),\nonumber \\
&f_{CLP}(u,v)=(v,u^3+uv^2,3u^4+2u^2v^2),\quad f_{D_4^{\pm}}(u,v)=(uv,u^2\pm 3v^2,u^2v\pm v^3),\nonumber 
\end{align}
where $k$ is a positive integer. The map germs $f^{(2,2k+1)}_{CE}$, $f_{SW}$, $f_{CCR}$, $f_{CBF}$, $f_{CS_1^{\pm}}$, $f_{CBK}$, $f_{CL}$, $f_{D_4^{\pm}}$ are called the {\it $(2,2k+1)$-cuspidal edge}, {\it swallowtail}, {\it cuspidal cross cap}, {\it cuspidal butterfly}, {\it cuspidal $S_1^{\pm}$ singularity}, {\it cuspidal beaks}, {\it cuspidal lips}, {\it $D_4^{\pm}$ singularity}, respectively (See Figure \ref{Fig1}).
\begin{figure}[htbp]
\begin{center}
\hspace{-1.6cm}\includegraphics[clip,scale=0.26,bb=0 0 1800 1200]{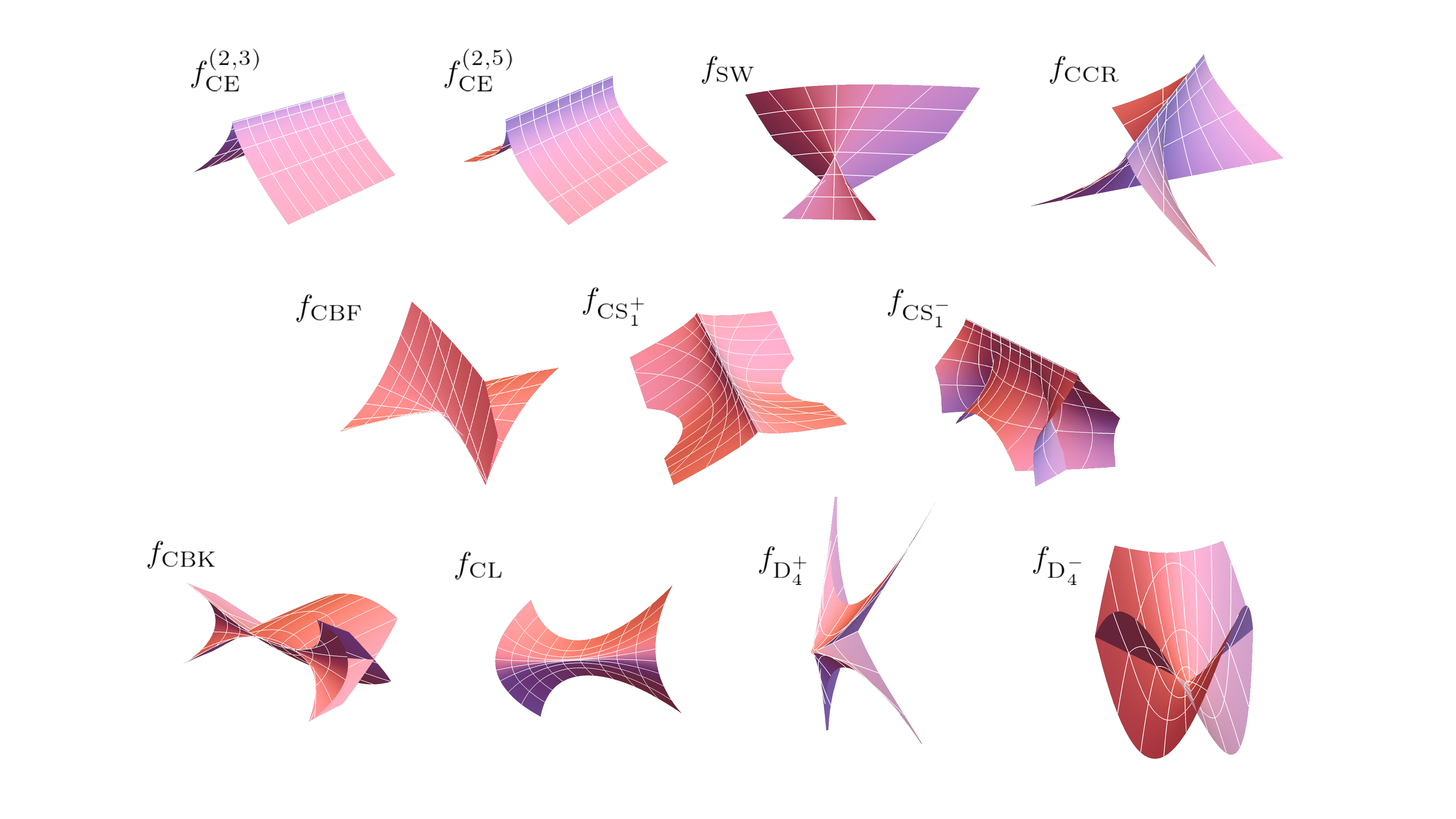}
\vspace{-0.8cm}
\caption{From top left to bottom right: the singularities of the map germs $f^{(2,3)}_{CE}$, $f^{(2,5)}_{CE}$, $f_{SW}$, $f_{CCR}$, $f_{CBF}$, $f_{CS_1^{+}}$, $f_{CS_1^{-}}$, $f_{CBK}$, $f_{CL}$, $f_{D_4^{+}}$, $f_{D_4^{-}}$.}\label{Fig1}
\end{center}
\end{figure}


\subsection{Timelike minimal surfaces and their singularities}
In this subsection, we clarify the notion of timelike minimal surfaces with singularities used here. First, we recall regular timelike minimal surfaces. 

A smooth immersion $f$ from a 2-dimensional smooth manifold $\Sigma$ to Minkowski $3$-space $\mathbb{L}^3=\left(\mathbb{R}^3,\langle,\rangle=-dt^2+dx^2+dy^2\right)$ is said to be {\it timelike} if its induced metric $\mathrm{I}=\hat{f}\langle, \rangle$ is Lorentzian. It is well-known that near each point of $\Sigma$, there exists a coordinate system $(U;u,v)$, called a {\it null coordinate system} such that $\mathrm{I}=2Fd{u}d{v}$, where $F$ is a smooth function. On $U$, $f$ is a {\it timelike minimal surface}, that is, whose mean curvature vanishes identically if and only if it satisfies the linear wave equation $f_{uv}=0$. Hence, $f$ can be locally decomposed into the sum of two regular null curves $\varphi$ and $\psi$:
\begin{equation}\label{eq:null_decomp0}
f(u,v)=\frac{\varphi(u)+\psi(v)}{2},
\end{equation}
where a {\it null curve} is a curve whose tangent vector field is lightlike. Conversely, if $f$ is written as \eqref{eq:null_decomp0} using two linearly independent regular null curves $\varphi$ and $\psi$, it is also a timelike minimal surface (cf. \cite{McNertney}).

 As well as maximal surfaces in $\mathbb{L}^3$, timelike minimal surfaces have singularities in general. To deal with such singularities, Kim, Koh, Shin and Yang \cite{KKSY}  and Takahasi \cite{T} (see also \cite{Akamine}) enlarged the class of timelike minimal surface as follows.
\begin{definition}\label{def:gtms}
A non-constant smooth map $f\colon\Sigma \longrightarrow \mathbb{L}^3$ is called a {\it generalized timelike minimal surface} if  $f$ is immersed on an open dense set of $\Sigma$ and there exists a local coordinate system $(U;u,v)$ near each point of $\Sigma$ such that $\langle f_u, f_u \rangle = \langle f_v, f_v \rangle =0$, and $f_{uv}= 0$ on $U$.
Moreover, if both $f_u$ and $f_v$ never vanish, $f$ is called a {\it minface}.
\end{definition}

The notion of minfaces is a timelike analogue of maxfaces introduced in \cite{UY} as a notion of spacelike maximal surfaces with rank one singularities. 

By definition, a generalized timelike minimal surface $f$ is also written as \eqref{eq:null_decomp0} although two null curves in  \eqref{eq:null_decomp0} may not be linearly independent or not be regular. 
Since the generating curves $\varphi$ and $\psi$ in \eqref{eq:null_decomp0} are null, $f$ can be locally written as follows:
\begin{equation}\label{eq:null_decomp}
\begin{split}
  f(u,v)
  &=\frac{1}{2}\int^u_{u_0}\left( -1-(g_1)^2,1-(g_1)^2, 2g_1 \right)\omega_1\\
  &+\frac{1}{2}\int^v_{v_0}\left( 1+(g_2)^2,1-(g_2)^2, -2g_2 \right)\omega_2 
    +f(u_0,v_0),
    \end{split}
\end{equation}
where $g_1(u)$ and $g_2(v)$ are functions, and $1$-forms $\omega_1(u)=\hat{\omega}_1(u)du$ and $\omega_2(v)=\hat{\omega}_2(v)dv$. The functions $g_1$ and $g_2$ are $\mathbb{R}\cup\{\pm \infty\}$-valued. The quadruple $(g_1, g_2,\omega_1, \omega_2)$ is called the {\it real Weierstrass data} (or {\it W-data} for short) of $f$.
By \eqref{eq:null_decomp}, W-data satisfies that if $g_i(p)=\pm \infty$ for some $i=1,2$ then $\hat{\omega_i}(p)=g_i\hat{\omega_i}(p)=0$ and $g_i^2\hat{\omega}_i(p)$ is a real number.  

To avoid unnecessary complications, throughout this paper we always assume that 
\begin{center}
 $g_1$ and $g_2$ have finite values at any singular point. 
\end{center}
This assumption holds automatically for any minface after taking an ambient isometry (see \cite[Fact A.7]{Akamine}).

The set of singular points, denoted by $\Sigma(f)$, of a generalized timelike minimal surface $f$ in \eqref{eq:null_decomp} is written as $\Sigma(f)=\Sigma^g(f)\cup \Sigma^\omega(f)$, where
\begin{equation}\label{eq:singular_set}
\Sigma^g(f)=\{p\in \Sigma \mid g_1g_2(p)=1\},\quad \Sigma^\omega(f)=\{p\in \Sigma \mid \hat{\omega}_1\hat{\omega}_2(p)=0\}.
\end{equation}
We call a singular point in $\Sigma^g(f)$ (resp.~ in $\Sigma^\omega(f)$) a {\it $g$-singular point} (resp.~{\it $\omega$-singular point}). 
The velocity vector fields of the generating null curves $\varphi$ and $\psi$ in \eqref{eq:null_decomp} are linearly dependent on $\Sigma^g(f)$, while at least one of $\varphi$ and $\psi$ is not regular on $\Sigma^\omega(f)$. Therefore, $\Sigma^g(f)\cap \Sigma^\omega(f) \neq  \emptyset$ in general, and it is important when considering the types of singularities whether they belong to $\Sigma^g(f)\cap \Sigma^\omega(f)$ or not as we will see in Section \ref{sec:omega_sing}. By Definition \ref{def:gtms}, a minface is a timelike minimal surface without $\omega$-singular points. 

We should remark that any generalized timelike minimal surface $f$ becomes a frontal with the unit normal vector field
\begin{equation}\label{eq:unitnormal}
n=\dfrac{1}{|n|_E}(g_1+g_2,-g_1+g_2,1+g_1g_2),\quad |n|_E=\sqrt{(1-g_1g_2)^2+2(g_1+g_2)^2}\neq0.
\end{equation}

\section{Singularities in $\Sigma^g(f)\setminus \Sigma^\omega(f)$}\label{Sec:Gsing}
In this section, we consider timelike minimal surfaces with singularities in $\Sigma^g(f)\setminus \Sigma^\omega(f)$, such surfaces are nothing but minfaces in Definition \ref{def:gtms}. First, we recall criteria for cuspidal edge, swallowtail and cuspidal cross cap on minfaces given by Takahashi \cite{T}.

\begin{fact}[\cite{T}, see also \cite{Akamine}]\label{Fact: T-lemma}
Let $f\colon U \longrightarrow \mathbb{L}^3$ be a minface with W-data $(g_1, g_2,\hat{\omega}_1du, \hat{\omega}_2dv)$ on $U$, and $p\in U$ be a singular point. Then $f$ is $\mathcal{A}$-equivalent to 
\begin{itemize}
\item[(i)] the cuspidal edge at $p$ if and only if $\frac{(g_1)_u}{g_1^2\hat{\omega}_1}-\frac{(g_2)_v}{g_2^2\hat{\omega}_2}\neq 0\text{ and }\frac{(g_1)_u}{g_1^2\hat{\omega}_1}+\frac{(g_2)_v}{g_2^2\hat{\omega}_2}\neq0\text{ at } p,$

\item[(ii)] the swallowtail at $p$ if and only if $$\frac{(g_1)_u}{g_1^2\hat{\omega}_1}-\frac{(g_2)_v}{g_2^2\hat{\omega}_2}\neq 0,\ \frac{(g_1)_u}{g_1^2\hat{\omega}_1}+\frac{(g_2)_v}{g_2^2\hat{\omega}_2}=0,\ \text{and}\ \frac{g_1}{(g_1)_u}\left(\frac{(g_1)_u}{g_1^2\hat{\omega}_1}\right)_u-\frac{g_2}{(g_2)_v} \left(\frac{(g_2)_v}{g_2^2\hat{\omega}_2}\right)_v\neq 0\text{ at }p,$$

\item[(iii)] the cuspidal cross cap at $p$ if and only if $$\frac{(g_1)_u}{g_1^2\hat{\omega}_1}-\frac{(g_2)_v}{g_2^2\hat{\omega}_2}=0,\ \frac{(g_1)_u}{g_1^2\hat{\omega}_1}+\frac{(g_2)_v}{g_2^2\hat{\omega}_2}\neq0,\ \text{and}\ \frac{g_1}{(g_1)_u}\left(\frac{(g_1)_u}{g_1^2\hat{\omega}_1}\right)_u+\frac{g_2}{(g_2)_v} \left(\frac{(g_2)_v}{g_2^2\hat{\omega}_2}\right)_v\neq0\text{ at }p.$$
\end{itemize}
\end{fact}

\begin{remark}
Here, we use slightly different but equivalent criteria to those in \cite{Akamine}.
\end{remark}

As shown in \cite{FSUY}, cuspidal edges, swallowtails and cuspidal cross caps are generic singularities on maximal surfaces and Fact \ref{Fact: T-lemma} shows that they also appear naturally on timelike minimal surfaces. However, as shown below, the appearance of other singularities is somewhat different from that of maximal surfaces. 

\subsection{Cuspidal butterfly and cuspidal $S_1^\pm$ singularity}
In this subsection, we give criteria for cuspidal butterfly and cuspidal $S_1^+$ singularity which appear in more complicated situations than those of the above singularities. Criteria for cuspidal buttefly and cuspidal $S_1^\pm$ singularity on frontals were given by Izumiya and Saji \cite{IS} and Saji \cite{S1} as follows.

\begin{fact}[\cite{IS}]\label{Butterfly_criteria}
Let $f\colon U\subset\mathbb{R}^2 \longrightarrow \mathbb{R}^3$ be a front with the unit normal $n$ and $p\in U$ be a non-degenerate singular point of $f$. Take a singular curve $\gamma(t)$ with $\gamma(0)=p$ and a null vector field $\eta(t)$.
 Then, $f$ is $\mathcal{A}$-equivalent to the cuspidal butterfly at $p$ if and only if $\delta(0)=\delta'(0)=0$ and $\delta''(0)\neq0$, where $\delta$ is the function defined in \eqref{eq:delta_psi}.
\end{fact}

\begin{fact}[\cite{S1}]\label{cS_criteria}
Let $f\colon U\subset\mathbb{R}^2 \longrightarrow \mathbb{R}^3$ be a frontal with the unit normal $n$ and $p\in U$ be a non-degenerate singular point of $f$. Take a singular curve $\gamma(t)$ with $\gamma(0)=p$ and a null vector field $\eta(t)$.
 Then, $f$ is $\mathcal{A}$-equivalent to the cuspidal $S_1^+$ singularity (resp.~the cuspidal $S_1^-$ singularity) at $p$ if and only if the followoing conditions hold.
 \begin{itemize}
 \item[(i)] $\delta(0)\neq 0$,
 \item[(ii)] $\psi(0)=\psi'(0)=0$ and $A:=\psi''(0)\neq 0$,
 \item[(iii)] there exists a regular curve $c\colon (-\varepsilon,\varepsilon)\longrightarrow U$ and $l\in \mathbb{R}$ such that $c(0)=p$, $c'(0)$ is parallel to $\eta(0)$, $\hat{c}''(0)\neq 0$, $\hat{c}'''(0)=\ell \hat{c}''(0)$ and
 \[
 B:=\det{\left(df(\gamma'(0)), \hat{c}''(0), 3\hat{c}^{(5)}-10\ell \hat{c}^{(4)}\right)} \neq0
 \] hold, where $\hat{c}:=f\circ c$,
 \item[(iv)] the product $AB$ is positive (resp.~negative). 
 \end{itemize}
\end{fact}

To use Fact \ref{cS_criteria} for timelike minimal surfaces, we use the following Lemma.

\begin{lemma}\label{Lem:Lemma35_OT}
Let $f\colon \Sigma \to \mathbb{R}^3$ be a frontal and $p\in \Sigma$ be a non-degenerate singular point. Let $\gamma=\gamma(t)$ be a singular curve with $\gamma(0)=p$ and $\xi, \eta$ be smooth vector fields near $p$ so that $\xi=\gamma', df(\eta)=0$ and $\det{(\xi,\eta)}>0$ along $\gamma$. Assume that $f$ satisfies the following conditions
\begin{itemize}
\item $\eta^3f(p)=\bm{0}$, 
\item $\xi f(p)$ and $\eta^2f(p)$ are linearly independent, and
\item $\xi\eta^3f(p)$ and $\eta^2f(p)$ are linearly dependent.
\end{itemize}
Then there exists a regular curve $c=c(t)$ satisfying the condition (iii) of Fact \ref{cS_criteria} with $B=3\det{(\xi f, \eta^2f, \eta^5f)}(p)$, where $\eta^k f$ is the k-times derivative $\eta \cdots \eta f$.
\end{lemma}

The proof of Lemma \ref{Lem:Lemma35_OT} will be given in Appendix \ref{sec:app}.

\begin{remark}\label{remark:OT}
The proof is similar to the case of maximal surfaces in \cite{OT}. However, the condition $\langle \xi f(p), \eta^2f(p) \rangle_E =0$ in \cite[Lemma 35]{OT} does not hold in general for the case of timelike minimal surfaces. So we need other assumptions as in Lemma \ref{Lem:Lemma35_OT} showing the same conclusion.
\end{remark}

By using the criteria above, we prove the following characterizations of the cuspidal butterfly and the cuspidal $S_1^+$ singularity on timelike minimal surfaces.

\begin{theorem}\label{Thm:CBF_CS1}
Let $f$ and $p$ be as in Fact \ref{Fact: T-lemma}.  Then $f$ is $\mathcal{A}$-equivalent to 
\begin{itemize}
\item[(i)] the cuspidal butterfly at $p$ if and only if 
\begin{align*}
\frac{(g_1)_u}{g_1^2\hat{\omega}_1}-\frac{(g_2)_v}{g_2^2\hat{\omega}_2}&\neq 0,\quad \frac{(g_1)_u}{g_1^2\hat{\omega}_1}+\frac{(g_2)_v}{g_2^2\hat{\omega}_2}=0,\quad
\frac{g_1}{(g_1)_u}\left(\frac{(g_1)_u}{g_1^2\hat{\omega}_1}\right)_u-\frac{g_2}{(g_2)_v} \left(\frac{(g_2)_v}{g_2^2\hat{\omega}_2}\right)_v= 0,\ \text{and} \\
&\frac{g_1}{(g_1)_u}\left[\frac{g_1}{(g_1)_u}\left(\frac{(g_1)_u}{g_1^2\hat{\omega}_1}\right)_u\right]_u + \frac{g_2}{(g_2)_v} \left[\frac{g_2}{(g_2)_v} \left(\frac{(g_2)_v}{g_2^2\hat{\omega}_2}\right)_v\right]_v\neq 0
\text{ at }p,
\end{align*}

\item[(ii)] the cuspidal $S_1^+$ singularity at $p$ if and only if 
\begin{align*}
\frac{(g_1)_u}{g_1^2\hat{\omega}_1}-\frac{(g_2)_v}{g_2^2\hat{\omega}_2}&=0,\quad \frac{(g_1)_u}{g_1^2\hat{\omega}_1}+\frac{(g_2)_v}{g_2^2\hat{\omega}_2}\neq0,\quad
\frac{g_1}{(g_1)_u}\left(\frac{(g_1)_u}{g_1^2\hat{\omega}_1}\right)_u+\frac{g_2}{(g_2)_v} \left(\frac{(g_2)_v}{g_2^2\hat{\omega}_2}\right)_v= 0,\ \text{and} \\
&\frac{g_1}{(g_1)_u}\left[\frac{g_1}{(g_1)_u}\left(\frac{(g_1)_u}{g_1^2\hat{\omega}_1}\right)_u\right]_u - \frac{g_2}{(g_2)_v} \left[\frac{g_2}{(g_2)_v} \left(\frac{(g_2)_v}{g_2^2\hat{\omega}_2}\right)_v\right]_v\neq 0
\text{ at }p,
\end{align*}
\end{itemize}
Moreover, there is no cuspidal $S_1^-$ singularity on any minface.
\end{theorem}

\begin{proof}[Proof of Theorem \ref{Thm:CBF_CS1}]
First, we prove the claim (i). Since $p$ is a non-degenerate $g$-singular point, we can take a singular curve $\gamma=\gamma(t)$ satisfying $\gamma(0)=p$ and $g_1(\gamma(t))g_2(\gamma(t))=1$. Hence we can parametrize the singular curve $\gamma$ and take a null vector field $\eta$ as follows. 
\begin{equation}\label{eq:xi_eta}
\gamma'=\frac{(g_2)_v}{g_2}\partial_u-\frac{(g_1)_u}{g_1}\partial_v,
\quad
\eta=\frac{1}{g_1\hat{\omega}_1}\partial_u+\frac{1}{g_2\hat{\omega}_2}\partial_v
\quad \text{on $\gamma$}.
\end{equation}
The functions $\delta$ and $\psi$ defined in $\eqref{eq:delta_psi}$ are written as
\begin{equation}\label{eq:delta_psi}
\delta=\frac{(g_1)_u}{g_1^2\hat{\omega}_1}+\frac{(g_2)_v}{g_2^2\hat{\omega}_2},
\quad
\psi=\alpha \left(\frac{(g_1)_u}{g_1^2\hat{\omega}_1}-\frac{(g_2)_v}{g_2^2\hat{\omega}_2}\right)
\ \text{on $\gamma$},
\end{equation}
where $\alpha=-\hat{\omega}_1\hat{\omega}_2\delta/2$. For the function
\[
D:=\frac{g_1g_2}{(g_1)_u(g_2)_v}\delta'
=\frac{g_1}{(g_1)_u}\left(\frac{(g_1)_u}{g_1^2\hat{\omega}_1}\right)_u-\frac{g_2}{(g_2)_v}\left(\frac{(g_2)_v}{g_2^2\hat{\omega}_2}\right)_v,
\]
$\delta'(0)=0$ is equivalent to $D(0)=0$ and under this condition ${\delta}''(0)\neq0$ is also equivalent to ${D}'(0)\neq0$.
By (i) in Fact \ref{Fact: T-lemma} and considering the quantity
\[
\frac{g_1g_2}{(g_1)_u(g_2)_v}D'=\frac{g_1}{(g_1)_u}\left[\frac{g_1}{(g_1)_u}\left(\frac{(g_1)_u}{g_1^2\hat{\omega}_1}\right)_u\right]_u + \frac{g_2}{(g_2)_v} \left[\frac{g_2}{(g_2)_v} \left(\frac{(g_2)_v}{g_2^2\hat{\omega}_2}\right)_v\right]_v,
\]
we obtain the assertion (i). 

Next, we prove the assertion (ii).  Let us define the functions $\varphi_i$, $\phi_i$ and $\Phi_i$ ($i=1,2$) as follow.
\[
\varphi_i:=\frac{g_i'}{g_i^2\hat{\omega}_i},\quad \phi_i:=\frac{g_i}{g_i'}\varphi_i',\quad \Phi_i:=\frac{g_i}{g_i'}\phi_i',\quad i=1,2,
\]
where the prime means the derivative with respect to $u$ for $i=1$ and with respect to $v$ for $i=2$, respectively.
By \eqref{eq:delta_psi}, $\delta=\varphi_1+\varphi_2$ and $\psi=-\frac{1}{2}\hat{\omega}_1\hat{\omega}_2(\varphi_1+\varphi_2)(\varphi_1-\varphi_2)$ hold, and hence the conditions (i) and (ii) in Fact \ref{cS_criteria} can be rewritten as
\begin{equation}\label{eq:phicondition}
\varphi_1=\varphi_2 \neq 0, \quad \phi_1=-\phi_2,\quad  A=-\frac{1}{2}\hat{\omega}_1\hat{\omega}_2(g_1)_u^2(g_2)_u^2(\varphi_1+\varphi_2)(\Phi_1-\Phi_2)\neq 0.
\end{equation}
This means that the conditions (i) and (ii) in Fact \ref{cS_criteria} are equivalent to the desired condition (ii) in Theorem \ref{Thm:CBF_CS1}. Thus, we show that the remaining conditions (iii) and (iv) with $AB>0$ in Fact \ref{cS_criteria} follow automatically from the conditions (i) and (ii). 

By using $\xi=\gamma'$ and $\eta$ which are defined to satisfy equation \eqref{eq:xi_eta} and the relation $g_1g_2=1$ along $\gamma$, the vector fields $\xi f$, $\eta^kf$ ($1\leq k \leq 5$) and $\xi\eta^3f$ at $p$ can be computed as follows.
\begin{align*}
\xi f&=-\frac{\hat{\omega}_1\hat{\omega}_2}{2}(\varphi_1+\varphi_2)(g_1+g_2,g_1-g_2,-2),\quad \eta f=\bm{0},\\
\eta^2f &= \frac{1}{2}(\varphi_1+\varphi_2)(-g_1+g_2,-g_1-g_2,0),\quad \eta^3f = \bm{0},\\
\eta^4f &= \frac{1}{2}\left[\varphi_1(\varphi_1^2+\phi_1^2+\varphi_1\Phi_1) + \varphi_2(\varphi_2^2+\phi_2^2+\varphi_2\Phi_2)\right](-g_1+g_2, -g_1-g_2,0)\\
&-\frac{3}{2}(\varphi_1^2\phi_1+\varphi_2^2\phi_2)(g_1+g_2,g_1-g_2,0),\\
\eta^5f &=  -\frac{1}{2}\left[\varphi_1^2(\varphi_1^2+7\phi_1^2+4\varphi_1\Phi_1) - \varphi_2^2(\varphi_2^2+7\phi_2^2+4\varphi_2\Phi_2)\right](g_1+g_2, g_1-g_2,0)+ \alpha\eta^2f,\\
 \xi\eta^3f &= -\frac{3}{2}(g_1)_u(g_2)_v(\varphi_1\phi_1+\varphi_2\phi_2)(g_1+g_2, g_1-g_2,0)+\beta\eta^2f
\end{align*}
at $p$, where $\alpha$ and $\beta$ are real constants. Moreover, by using \eqref{eq:phicondition}, $\xi f, \eta^2f, \eta^5f$ and $\xi\eta^3f$ are written as
\begin{align*}
\xi f&=-\hat{\omega}_1\hat{\omega}_2\varphi_1(g_1+g_2,g_1-g_2,-2),\quad 
\eta^2f = \varphi_1(-g_1+g_2,-g_1-g_2,0), \\
\eta^5f &=  -2\varphi_1^3(\Phi_1-\Phi_2)(g_1+g_2, g_1-g_2,0)+ \alpha\eta^2f,\quad
 \xi\eta^3f =\beta\eta^2f.
\end{align*}
Therefore, $f$ satisfies the assumptions in Lemma \ref{Lem:Lemma35_OT} and hence $B$ is computed as
\begin{equation}\label{eq:B}
B=3\det{(\xi f, \eta^2f, \eta^5f)}(p)= 
-48\hat{\omega}_1\hat{\omega}_2\varphi_1^5(\Phi_1-\Phi_2).
\end{equation}
Finally, by \eqref{eq:phicondition} and \eqref{eq:B}, the quantity $AB$ in (iv) in Fact \ref{cS_criteria} satisfies 
\[
AB=48\hat{\omega}_1^2\hat{\omega}_2^2(g_1)_u^2(g_2)_u^2\varphi_1^6(\Phi_1-\Phi_2)^2>0,
\]
which proves the desired result.
\end{proof}

Let us see that certain singularities on timelike minimal surfaces are transferred to each other by the following anti-isometric transformation called the conjugation.

\begin{definition}\label{def:conjugate}
Let $(g_1, g_2,\hat{\omega}_1du, \hat{\omega}_2dv)$ be W-data of a generalized timelike minimal surface $f$. \begin{itemize}
\item The one-parameter family $\{f_\theta\}_{\theta \in \mathbb{R}}$ of generalized timelike minimal surafces with W-data $(g_1, g_2,e^\theta \hat{\omega}_1du, e^{-\theta}\hat{\omega}_2dv)$ is called the {\it associated family} of $f$. 
\item The generalized timelike minimal surface $\hat{f}$ with W-data $(g_1, g_2,\hat{\omega}_1du, -\hat{\omega}_2dv)$ is called the {\it conjugate surface} of $f$.
\end{itemize}
\end{definition}
The surfaces $f$ and $f_{\theta}$ are isometric each other. The surfaces $f$ and $\hat{f}$ are {\it anti-isometric} each other, that is, their first fundamental forms $\mathrm{I}_f$ and $\mathrm{I}_{\hat{f}}$ satisfy $\mathrm{I}_f=-\mathrm{I}_{\hat{f}}$. Hence, 
unlike the cases of minimal surfaces in $\mathbb{E}^3$ and maximal surfaces in $\mathbb{L}^3$, the conjugate surface of a timelike minimal surface does not belong to the associated family. Moreover, it is shown in \cite{Akamine2} that away from flat points the associated families $\{f_\theta \}$ of $f$ and $\{\hat{f}_\theta \}$ of $\hat{f}$ are the only isometric deformation and anti-isometric deformation of $f$, respectively, up to rigid motions in $\mathbb{L}^3$. 

By Definition \ref{def:conjugate} and Theorem \ref{Thm:CBF_CS1}, we obtain the following duality of singularities.
\begin{corollary}\label{cor:duality_CBF}
Let $f$ be a timelike minimal surface and $p$ be a singular point in $\Sigma^g(f)\setminus \Sigma^\omega(f)$. Then, $f$ is $\mathcal{A}$-equivalent to the cuspidal butterfly at $p$ if and only if $\hat{f}$ is $\mathcal{A}$-equivalent to the cuspidal $S_1^+$ singularity at $p$.
\end{corollary}

\begin{remark}
Ogata and Teramoto \cite{OT} proved the duality between cuspidal butterflies and cuspidal $S_1^-$ singularities on maximal surfaces. They also proved that there is no cuspidal $S_1^+$ singularity on these surfaces. In contrast, Theorem \ref{Thm:CBF_CS1} and Corollary  \ref{cor:duality_CBF} show a different aspect: cuspidal $S_1^+$ singularities appear on timelime minimal surfaces instead of cuspidal $S_1^-$ singularities.\end{remark}

\begin{example}
Let us consider the timelike minimal surface $f$ with W-data
\begin{align*}
g_1(u)=-\frac{\cos{u}}{1+\sin{u}},\ g_2(v)=\frac{\sinh{\phi}+\cos{v}\cosh{\phi}}{\sin{v}-\cosh{\phi}-\cos{v}\sinh{\phi}},\ \\
 \hat{\omega}_1(u)=-\frac{\mu}{2}(1+\sin{u}),\ \hat{\omega}_2(v)=\frac{\sin{v}-\cosh{\phi}-\cos{v}\sinh{\phi}}{2\mu},
\end{align*}
where $\phi=\arcsinh{(1/2)}$ and $\mu=(1+\sqrt{5})/2=\cosh{\phi}+\sinh{\phi}$.
Here, we show that the surface $f$ is $\mathcal{A}$-equivalent to the cuspidal butterfly at $(u,v)=\bm{0}$. Since $g_1g_2=1$ at $\bm{0}$ and the quantities in Theorem \ref{Thm:CBF_CS1} are computed as
\begin{align*}
\frac{(g_1)_u}{g_1^2\hat{\omega}_1}+&\frac{(g_2)_v}{g_2^2\hat{\omega}_2}
=
\frac{g_1}{(g_1)_u}\left(\frac{(g_1)_u}{g_1^2\hat{\omega}_1}\right)_u
=
\frac{g_2}{(g_2)_v} \left(\frac{(g_2)_v}{g_2^2\hat{\omega}_2}\right)_v
=0,\quad \frac{(g_1)_u}{g_1^2\hat{\omega}_1}+\frac{(g_2)_v}{g_2^2\hat{\omega}_2}=-4/\mu\neq0,& \\
&\text{ and } \frac{g_1}{(g_1)_u}\left[\frac{g_1}{(g_1)_u}\left(\frac{(g_1)_u}{g_1^2\hat{\omega}_1}\right)_u\right]_u + \frac{g_2}{(g_2)_v} \left[\frac{g_2}{(g_2)_v} \left(\frac{(g_2)_v}{g_2^2\hat{\omega}_2}\right)_v\right]_v=2\neq 0
\text{ at } \bm{0},&
\end{align*}
$f$ is $\mathcal{A}$-equivalent to the cuspidal butterfly at $\bm{0}$. By the duality in Corollary \ref{cor:duality_CBF}, the conjugate surface $\hat{f}$ is $\mathcal{A}$-equivalent to the cuspidal $S_1^+$ singularity at $\bm{0}$ (see Fig.~\ref{Fig:butterfly}).
We remark that the surface $f$ belongs to the associated family of a timelike minimal surface with planar curvature lines studied in \cite{ACO}. 
\end{example}
\vspace{2.0cm}
\begin{figure}[!h]
\begin{center}
\begin{tabular}{c}
\hspace{+2.0cm}
\begin{minipage}{0.4\hsize}
\begin{center}
\vspace{-2.8cm}
\includegraphics[clip,scale=0.45,bb=0 0 355 389]{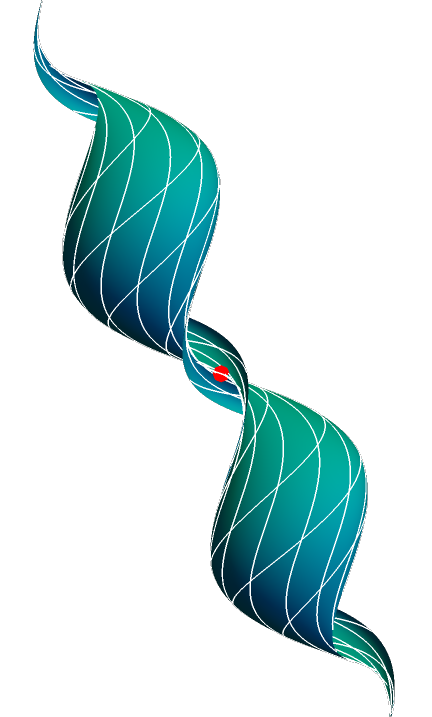}
\vspace{0.5cm}
\end{center}
\end{minipage}
\hspace{0.3cm}
\begin{minipage}{0.4\hsize}
\begin{center}
\vspace{-4.55cm}
\includegraphics[clip,scale=0.45,bb=0 0 355 509]{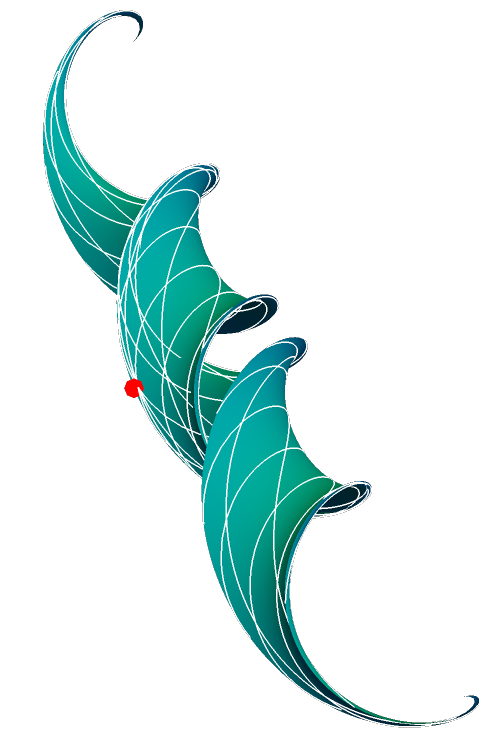}
\vspace{0.3cm}
\end{center}
\end{minipage}

\end{tabular}
\end{center}
\vspace{0.0cm}
\caption{Timelike minimal surfaces with a cuspidal butterfly (left) and with a cuspidal $S_1^+$ singularity (right). The red dots represent these singular points.}
 \label{Fig:butterfly}
\end{figure}

\section{Singularities in $\Sigma^\omega(f)$}\label{sec:omega_sing}
In this section, we give some properties of $\omega$-singular points on a generalized timelike minimal surface.
\begin{lemma}\label{lemma_front_condition}
Let $p\in \Sigma^\omega(f)$ be a $\omega$-singular point satisfying $\text{rank}(df_p)=1$ of a generalized timelike minimal surface $f$ with W-data $(g_1, g_2,\hat{\omega}_1du, \hat{\omega}_2dv)$. Then the following statements hold.
\begin{itemize}
\item[(i)] $f$ is a front at  $p$ with $\hat{\omega}_1(p)=0$ (resp.\ with $\hat{\omega}_2(p)=0$) if and only if $(g_1)_u\neq 0$ (resp.\ $(g_2)_v\neq 0$) at $p$.
\item[(ii)] The singular point $p$ with $\hat{\omega}_1(p)=0$ (resp.\ with $\hat{\omega}_2(p)=0$) is non-degenerate if and only if $(\hat{\omega}_1)_u(p)\neq 0$ (resp.\ $(\hat{\omega}_2)_v(p)\neq 0$) and $g_1g_2(p)\neq1$.
\end{itemize}
\end{lemma}

\begin{proof}
By taking the derivatives of the equation \eqref{eq:unitnormal}, we have 
\begin{align}
n_u=(n_u^1,n_u^2,n_u^3)=\frac{(g_1)_u}{|n|_E^3}\left((1+g_2^2)(1-g_1g_2),-1-3g_1g_2-3g_2^2-g_1g_2^3,2(g_2^3-g_1)\right), \label{eq:n_u} \\
n_v=(n_v^1,n_v^2,n_v^3)=\frac{(g_2)_v}{|n|_E^3}\left((1+g_1^2)(1-g_1g_2),1+3g_1g_2+3g_1^2+g_1^3g_2,2(g_1^3-g_2)\right) \label{eq:n_v}.
\end{align}
Since $f_u(p)=0$ for $p\in \Sigma^\omega(f)$ with $\hat{\omega}_1(p)=0$, $f$ is a front at $p$ if and only if $n_u(p)\neq0$. By \eqref{eq:n_u}, $n_u(p)\neq0$ is equivalent to the condition $(g_1)_u(p)=0$. Then, we have proved (i) for the case $\hat{\omega}_1(p)=0$. Finally, the signed area density function $\lambda$ can be written as
\begin{equation}\label{eq:lambda}
\lambda=\Lambda(1-g_1g_2)\hat{\omega}_1\hat{\omega}_2,\quad \Lambda=-\sqrt{(1-g_1g_2)^2+2(g_1+g_2)^2}/2.
\end{equation}
Since $\hat{\omega}_1(p)=0$ we have $d\lambda_p=\Lambda(p)(1-g_1g_2(p))(\hat{\omega}_1)_u(p)\hat{\omega}_2(p)$, and hence $p$ is non-degenerate if and only if $(\hat{\omega}_1)_u\neq 0$ and $g_1g_2\neq1$ at $p$. The proof of (i) and (ii) for the case $\hat{\omega}_2(p)=0$ is same.
\end{proof}

In the following sections, when considering an $\omega$-singular point $p$ satisfying $\text{rank}(df_p)=1$, we will state the assertions only for the case $\hat{\omega}_2(p)=0$.

\subsection{$(2,2k+1)$-cuspidal edges}
In this section, we consider $(2,2k+1)$-cuspidal edge of the form:
\begin{equation}\label{eq:2k+1cusp}
(u,v)\longmapsto (u,v^2,v^{2k+1}), \quad k=1,2,\cdots
\end{equation}
In general, if we consider a cuspidal edge of the form $f(u,v)=(u,v^m,v^n)$ $(2\leq m\leq n)$ and assume that it is non-degenerate, then we have $m=2$. Hence, \eqref{eq:2k+1cusp} means that we consider arbitrary non-degenerate cuspidal edges.

Usually the (2,3)-cuspidal edge is simply called the {\it cuspidal edge}. First, we give a characterization of cuspidal edges as $\omega$-singular points.

\begin{lemma}\label{lemma_first_kind}
The cuspidal edge is only non-degenerate rank one $\omega$-singular point on which the considered generalized timelike minimal surface is a front.
\end{lemma}

\begin{proof}
For an $\omega$-singular point $p=(0,0)$ with $\hat{\omega}_2(p)=0$ and $\mathrm{rank}(df_p)=1$, we can take the singular curve $\gamma(u)=(u,0)$ and the null direction given by $\eta=\frac{\partial}{\partial{v}}$, which satisfies $\det{(\gamma',\eta)}\neq0$. By the useful criterion proved in \cite[Proposition 1.3]{KRSUY}, we prove the desired result.
\end{proof}

By Lemma \ref{lemma_first_kind}, any rank one non-degenerate $\omega$-singular point on which the considered surface is a front becomes a cuspidal edge automatically. On the other hand, it is natural to ask what is a typical rank one non-degenerate $\omega$-singular point on which the considered surface is not a front. As an answer to this question, we prove the following theorem.

\begin{theorem}\label{criterion_(2,5)CE}
Let $p\in \Sigma^\omega(f)$ be a $\omega$-singular point satisfying $\text{rank}(df_p)=1$ and $\hat{\omega}_2(p)=0$ of a generalized timelike minimal surface $f=f(u,v)$ with W-data $(g_1, g_2,\hat{\omega}_1du, \hat{\omega}_2dv)$. If $f$ is not a front at $p$, then $f$ is $\mathcal{A}$-equivalent to the $(2,5)$-cuspidal edge if and only if $\left(\frac{(g_2)_{vv}}{(\hat{\omega}_2)_v}\right)_v(p)\neq0. 
$
\end{theorem}

To prove Theorem \ref{criterion_(2,5)CE}, we use the following criterion for $(2,5)$-cuspidal edges.
\begin{fact}[\cite{HKS}]\label{HKS_criterion}
Let $U$ be a domain of $\mathbb{R}^2$,
$f\colon U \rightarrow \mathbb{R}^3$ a frontal,
and $p \in U$ a non-degenerate singular point of the first kind.
Moreover, let
$\gamma(t)$ $(|t|<\varepsilon)$ 
be a singular curve passing through $p=\gamma(0)$
and $\eta(t)$ a null vector field along $\gamma$.
Take smooth vector fields $\xi=\xi(u,v)$ and $\eta=\eta(u,v)$ on $U$
which are extensions of $\gamma'(t)$ and $\eta(t)$, respectively.
Then, $p=\gamma(0)$ is a $(2,5)$-cuspidal edge 
if and only if 
\begin{gather}
  \det(\xi f,\, \eta^2f,\, \eta^3f)=0
  \ \text{on $\gamma$}, \quad and  
  \label{eq:condition-3}\\
  \det(\xi f,\, \tilde\eta^2 f,\, 3\tilde\eta^5 f - 10 C \tilde\eta^4 f)(p)\neq0
  \label{eq:condition-4}
\end{gather}
hold, where $\eta^k f$ implies $k$-times derivative $\eta \cdots \eta f$, 
$\tilde{\eta}$ is a special null vector field satisfying
\begin{equation}\label{eq:special-null}
  \langle\xi f, \tilde\eta^2 f\rangle_E(p)=
  \langle\xi f, \tilde\eta^3 f\rangle_E(p)=0
\end{equation}
{\rm (}$\langle, \rangle_E$ means the Euclidean inner product\/{\rm )},
and $C$ is a constant such that
\begin{equation}\label{eq:constant-C}
  \tilde\eta^3 f(p) = C \tilde\eta^2 f(p).
\end{equation}
\end{fact}

Now we can take the vector field $\tilde{\eta}$ and the constant $C$ as follows.
\begin{lemma}\label{lemma_aboutC}
Let $f$ and $p$ be the same as in Theorem \ref{criterion_(2,5)CE}, and assume that $p=\bm{0}$. Then we can take the vector filed $\tilde{\eta}$ and the constant $C$ in Fact \ref{HKS_criterion} as follows.
\begin{equation*}
\tilde{\eta}(u,v)=(av+bv^2) \frac{\partial}{\partial u}+\frac{\partial}{\partial v}\quad \text{ and }\quad C=\frac{(\hat{\omega}_2)_{vv}}{(\hat{\omega}_2)_v}(p),
\end{equation*}
where the constants $a$ and $b$ are
\begin{equation}\label{eq:a_b}
a=\frac{(\hat{\omega}_2)_v}{\hat{\omega}_1}\left(\frac{g_1+g_2}{1+g_1^2}\right)^2(p) \quad\text{ and }\quad b=\frac{(\hat{\omega}_2)_{vv}}{2\hat{\omega}_1}\left(\frac{g_1+g_2}{1+g_1^2}\right)^2(p).
\end{equation}
\end{lemma}

\begin{proof}
As proved in \cite[Lemma 4.2]{HKS}, we can construct $\tilde{\eta}$ satisfying the condition (\ref{eq:special-null}) as 
\[
\tilde{\eta}(u,v)=(av+bv^2) \frac{\partial}{\partial u}+\frac{\partial}{\partial v},\text{ where }a=-\frac{\langle f_u,f_{vv}\rangle_E}{\langle f_u,f_u\rangle_E}(p)\text{ and } b=-\frac{\langle f_u,f_{vvv}\rangle_E}{2\langle f_u,f_u\rangle_E}(p).
\]
By the assumption and Lemma \ref{lemma_front_condition}, the relations $\hat{\omega}_2(p)=g_2'(p)=0$ hold and hence we can calculate 
\begin{align*}
f_u=\frac{\hat{\omega}_1}{2}&(-1-g_1^2, 1-g_1^2, 2g_1),\\
f_{vv}=\frac{(\hat{\omega}_2)_v}{2}(1+g_2^2,1-g_2^2, -2g_2),&\quad f_{vvv}=\frac{(\hat{\omega}_2)_{vv}}{2}(1+g_2^2,1-g_2^2, -2g_2)\quad \text{at $p$.}
\end{align*}
By using them, 
\begin{align*}
\langle f_u,f_u\rangle_E&=\frac{\hat{\omega}_1}{2}(1+g_1^2)^2,\\
\langle f_u,f_{vv}\rangle_E=-\frac{\hat{\omega}_1(\hat{\omega}_2)_v}{2}(g_1+g_2)^2,&\quad
\langle f_u,f_{vvv}\rangle_E=-\frac{\hat{\omega}_1(\hat{\omega}_2)_{vv}}{2}(g_1+g_2)^2
\end{align*}
hold at $p$. Thus, the constants $a$ and $b$ are written as (\ref{eq:a_b}).

For $\tilde{\eta}=(av+bv^2) \frac{\partial}{\partial u}+\frac{\partial}{\partial v}$, $\tilde{\eta}^2{f}(p)$ and $\tilde{\eta}^3{f}(p)$ are written as
\begin{equation}
{\tilde{\eta}^2f}=af_u+f_{vv},\quad \tilde{\eta}^3{f}=2bf_u+f_{vvv}\quad \text{at $p$}.
\end{equation}
Based on the relations 
\begin{equation*}
f_{vvv}(p)=\frac{(\hat{\omega}_2)_{vv}}{(\hat{\omega}_2)_v}f_{vv}(p)\quad \text{and} \quad 2b=\frac{(\hat{\omega}_2)_{vv}}{(\hat{\omega}_2)_v}a,
\end{equation*}
we obtain the desired constant $C$.
\end{proof}

By using Fact \ref{HKS_criterion} and Lemma \ref{lemma_aboutC}, we can prove Theorem \ref{criterion_(2,5)CE} as follows.
\begin{proof}[Proof of Theorem \ref{criterion_(2,5)CE}]
Without loss of generality, we may assume that $p=(0,0)$.
First we check the condition \eqref{eq:condition-3}. Since the singular curve passing through $p$ is $\gamma(u)=(u,0)$, we can take the vector fields $\xi$ and $\eta$ as $\xi=\partial/\partial u$ and $\eta=\partial/\partial v$.
As we saw in the proof of Lemma \ref{lemma_aboutC}, $\eta^2f=f_{vv}$ and $\eta^3f=f_{vvv}$ are parallel on $\gamma$. Therefore we have checked the condition \eqref{eq:condition-3}. 

Next, we see the quantity \eqref{eq:condition-4}.
By using Lemma \ref{lemma_aboutC}, we can calculate that
\begin{align*}
{\tilde{\eta}^4f}&=3a^2f_{uu}+f_{vvvv} \\
&=3a^2\left[\frac{(\hat{\omega}_1)_u}{2}\bm{c_1}-\hat{\omega}_1(g_1)_u\bm{c_3}\right]
+\frac{(\hat{\omega}_2)_{vvv}}{2}\bm{c_2}-3(\hat{\omega}_2)_v(g_2)_{vv}\bm{c_4},\\
\tilde{\eta}^5{f}=&20abf_u+f_{vvvvv}\\
=&20ab\left[\frac{(\hat{\omega}_1)_u}{2}\bm{c_1}-\hat{\omega}_1(g_1)_u\bm{c_3}\right]+\frac{(\hat{\omega}_2)_{vvvv}}{2}\bm{c_2}-\left[6(\hat{\omega}_2)_{vv}(g_2)_{vv}+4(\hat{\omega}_2)_v(g_2)_{vvv}\right]\bm{c_4}
\end{align*}
at $p$, where
\begin{gather}
\bm{c_1}=\left(-1-g_1^2(p),1-g_1^2(p),2g_1(p)\right),\quad \bm{c_2}=\left(1+g_2^2(p),1-g_2^2(p),-2g_2(p)\right), \nonumber \\
\bm{c_3}=\left(g_1(p),g_1(p),-1\right),\quad \bm{c_4}=\left(-g_2(p),g_2(p),1\right).\nonumber
\end{gather}
Nothing that $\xi$ and $\tilde{\eta}^2f$ at $p$ were calculated in the proof of Lemma \ref{lemma_aboutC} as follows
\[
\xi f=\frac{\hat{\omega}_1}{2}\bm{c_1}\quad \text{ and }\quad
\tilde{\eta}^2f=a\frac{(\hat{\omega}_1)}{2}\bm{c_1}+\frac{(\hat{\omega}_2)_v}{2}\bm{c_2},
\]
we obtain the following equations at $p$
\begin{align*}
&\det{\left(\xi{f},\tilde{\eta}^2{f},-10C\tilde{\eta}^4{f} \right)}=-\frac{5C\hat{\omega}_1(\hat{\omega}_2)_v}{2}\det{\left(\bm{c_1},\bm{c_2},\tilde{\eta}^4{f} \right)}\nonumber \\
&=\frac{5\hat{\omega}_1(\hat{\omega}_2)_{vv}}{2}\left[3a^2 \hat{\omega}_1(g_1)_u \det{\left(\bm{c_1},\bm{c_2},\bm{c_3} \right)}+3(\hat{\omega}_2)_v(g_2)_{vv}\det{\left(\bm{c_1},\bm{c_2},\bm{c_4} \right)}\right]\nonumber \\
&=15a^2(\hat{\omega}_1)^2 (\hat{\omega}_2)_{vv} (g_1)_u (1-g_1g_2)^2 - 15\hat{\omega}_1(\hat{\omega}_2)_v (\hat{\omega}_2)_{vv} (g_2)_{vv} (1-g_1g_2)^2,
\end{align*}
and 
\begin{align*}
\det{\left(\xi{f},\tilde{\eta}^2{f},3\tilde{\eta}^5{f} \right)}&=\frac{3\hat{\omega}_1(\hat{\omega}_2)_v}{4}\det{\left(\bm{c_1},\bm{c_2},\tilde{\eta}^5{f} \right)}&\nonumber \\
&=-\frac{3\hat{\omega}_1(\hat{\omega}_2)_{v}}{4}[20ab \hat{\omega}_1(g_1)_u \det{\left(\bm{c_1},\bm{c_2},\bm{c_3} \right)}&\nonumber\\
&\hspace{2.6cm} + \left( 6(\hat{\omega}_2)_{vv}(g_2)_{vv}+4(\hat{\omega}_2)_v(g_2)_{vvv}\right)\det{\left(\bm{c_1},\bm{c_2},\bm{c_4} \right)}]&\nonumber \\
&=-15a^2(\hat{\omega}_1)^2 (\hat{\omega}_2)_{vv} (g_1)_u (1-g_1g_2)^2\nonumber\\
 &\quad + \hat{\omega}_1(\hat{\omega}_2)_v \left( 9(\hat{\omega}_2)_{vv}(g_2)_{vv}+6(\hat{\omega}_2)_v(g_2)_{vvv}\right)(1-g_1g_2)^2,
\end{align*}
where we used the relations $2ab=a^2(\hat{\omega}_2)_{vv}/(\hat{\omega}_2)_{v}$, see Lemma \ref{lemma_aboutC}, and $\det{\left(\bm{c_1},\bm{c_2},\bm{c_3} \right)}=-\det{\left(\bm{c_1},\bm{c_2},\bm{c_4} \right)}=2(1-g_1g_2)^2$ in the last equality. Therefore, we have

\begin{align*}
 \det(\xi f,\, \tilde\eta^2f,\, 3\tilde\eta^5 f - 10 C \tilde\eta^4 f)=&6\hat{\omega}_1(\hat{\omega}_2)_v (1-g_1g_2)^2\left( (\hat{\omega}_2)_{vv}(g_2)_{vvv}-(\hat{\omega}_2)_{vv}(g_2)_{vv}\right)\\
 =&6\hat{\omega}_1(\hat{\omega}_2)_v^3 (1-g_1g_2)^2\left( \frac{(g_2)_{vv}}{(\hat{\omega}_2)_v}\right)_v\text{ at $p$.}
\end{align*}
Noting that $\hat{\omega}_1(\hat{\omega}_2)_v^3 (1-g_1g_2)^2\neq 0$ at $p$ due to non-degeneracy of $p$ and Lemma \ref{lemma_front_condition}, we have the desired result.
\end{proof}

As a corollary of Theorem \ref{criterion_(2,5)CE}, we obtain the following invariance property of (2,5)-cuspidal edges under the isometric and anti-deformtaions introduced in Definition \ref{def:conjugate}.
\begin{corollary}\label{cor:duality_(2,5)_cusp}
Let $p$ be an $\omega$-singular point of a generalized timelike minimal surface $f$. Then, the following statements are equivalent.
\begin{itemize}
\item[(i)] $f$ is $\mathcal{A}$-equivalent to the $(2,5)$-cuspidal edge at $p$,
\item[(ii)] $f_{\theta}$ is $\mathcal{A}$-equivalent to the $(2,5)$-cuspidal edge at $p$ for arbitrary $\theta\in \mathbb{R}$ and 
\item[(iii)] $\hat{f}$ is $\mathcal{A}$-equivalent to the $(2,5)$-cuspidal edge at $p$.
\end{itemize}
\end{corollary}

This is an interesting phenomenon by the following reasons:
Honda, Koiso and Saji \cite[Remark 4.9]{HKS} pointed out that maxfaces cannot have $(2,5)$-cuspidal edges because if we assume that a singular point $p$ on a maxface satisfies the condition of $(2,5)$-cuspidal edges, we can prove that $p$ becomes a fold singularity. On the other hand, they also proved in \cite[Theorems 1.1 and 1.2]{HKS} that spacelike surfaces with non-zero constant mean curvature (spacelike CMC surfaces, for short) do not admit fold singularities and there is a duality of conlike singularities of Delaunay surfaces and $(2,5)$-cuspidal edges of their conjugate surfaces. However Theorem \ref{criterion_(2,5)CE} and Corollary \ref{cor:duality_(2,5)_cusp} show that for the timelike case, $(2,5)$-cuspidal edges appear frequently in the sense of Theorem \ref{criterion_(2,5)CE} even for zero mean curvature surfaces, and such singular points have a self duality which is different from that of spacelike CMC surfaces.
Moreover, in contrast to the case of maximal surfaces, the following proposition shows that timelike minimal surfaces can have any type of $(2,2k+1)$-cuspidal edges, that is, cuspidal edges can be as steep as desired.

\begin{proposition}\label{ExistenceCusp}
For any positive integer $k$, the timelike minimal surface $f$ with W-data $(g_1, g_2,\hat{\omega}_1du, \hat{\omega}_2dv)$ satisfying
$g_2(v)=v^{2k-1}$, $\hat{\omega}_1(0)\neq0$ and $\hat{\omega}_2(v)=v$,
has $(2,2k+1)$-cuspidal edges on $v=0$.
\end{proposition}

\begin{proof}
After a translation and the local diffeomorphism  
$\Phi_1(X,Y,Z)=(-X+Y,Y,Z)$,
the timelike minimal surface $f$ defined by \eqref{eq:null_decomp} is written by
\begin{equation*}
f_1(u,v)= \Phi_1^{-1}\circ f(u,v)=\frac{1}{2}\int^u_{0}\left( 2,1-(g_1)^2, 2g_1 \right)\omega_1+\left( -\frac{v^{4k}}{4k},\frac{v^{2}}{4}-\frac{v^{4k}}{8k},\frac{v^{2k+1}}{2k+1}\right).
    \end{equation*}
  We set $(\gamma^0,\gamma^1,\gamma^2)=\left( -\frac{v^{4k}}{4k},\frac{v^{2}}{4}-\frac{v^{4k}}{8k},\frac{v^{2k+1}}{2k+1}\right) $ and consider the new parameter $V=v\sigma(v)$, $\sigma(v)=(\frac{1}{4}-\frac{v^{4k-2}}{8k})^{\frac{1}{2}}$ near $v=0$.
 Since 
 \[
 \frac{d{\gamma^0}}{dV}(V)=-v^{4k-1}\tau(v),\quad \tau(v)=\frac{d{v}}{dV}=(\sigma(v)+v\sigma'(v))^{-1},
 \]
 we obtain the Taylor expansion $ \gamma^0(V)=\sum^\infty_{i=4k}a_iV^i, a_{4k}\neq 0$.

 By separating the even and odd terms, we can take functions $h_1$ and $h_2$ satisfying $h_1(0)=h_1'(0)=h_2(0)=0$ such that $\gamma^0(V)=h_1(V^2)+V^{2k+1}h_2(V^2)$. Similarly, $\gamma^2$ is also expressed as
 \[
 \gamma^2(V)=b_{2k+1}V^{2k+1}+b_{2k+1}V^{2k+1}h_3(V^2)+h_4(V^2), 
 \]
 where $b_{2k+1}\neq 0$ and $h_3$, $h_4$ are functions satisfying $h_3(0)=h_4(0)=0$.
 
 Therefore, if we put the new parameter $U=\int^u_0\omega_1$, the map $f_1$ is written as
 \begin{align*}
f_1(U,V)&=(U,f^1_1(U),f^2_1(U))\\
&+\left(h_1(V^2)+V^{2k+1}h_2(V^2), V^2, b_{2k+1}V^{2k+1}+ b_{2k+1}V^{2k+1}h_3(V^2)+h_4(V^4) \right),
 \end{align*}
 where $f^1_1$ and $f^2_1$ are functions with $f^1_1(0)=f^2_1(0)=0$.
Since the map
 \[
 \Phi_2(X,Y,Z)=(X+h_1(Y)+Zh_2(Y),f^1_1(X)+Y,f^2_1(X)+b_{2k+1}Zh_3(Y)+h_4(Y)),
 \]
 satisfies $\Phi_2(0)=0$ and the Jacobian of $\Psi_2$ at $0$ is $b_{2k+1}-b_{2k+1}h_1'(0){(f^1_1)}'(0)=b_{2k+1}$, it is a local diffeomorphism of $\mathbb{R}^3$ at $0$. Hence we conclude that 
 \[
  \Phi_2^{-1}\circ f_1(U,V)=(U,V^2,V^{2k+1}).
 \]
\end{proof}

\subsection{Cuspidal beaks and lips} Curspidal beaks and cuspidal lips are degenerate rank one singular points of fronts,  and criteria for them were given by Izumiya, Saji and Takahasi \cite{IST} as follows.
\begin{fact}[\cite{IST}]\label{ISTcriteria}
Let $f\colon U\subset\mathbb{R}^2 \longrightarrow \mathbb{R}^3$ be a front and $p\in U$ be a singular point of $f$. The map $f$ is $\mathcal{A}$-equivalent to the cuspidal beaks (resp.\ cuspidal lips) at $p$ if and only if $\rm{rank}{(df_p)}=1$, $d\lambda_p=0$,  $\eta\eta\lambda(p)\neq0$ and $\det{(\Hess{\lambda)}}(p)<0$ and (resp.\ $\rm{rank}{(df_p)}=1$, $d\lambda_p=0$ and $\det{(\Hess{\lambda})}(p)>0$) hold, where $\det{(\Hess{\lambda})}$ means the determinant of the Hessian matrix of the function $\lambda$.
\end{fact}
Applying Fact \ref{ISTcriteria}, we have the following. 
\begin{proposition}\label{prop:BeaksLips}
Let $p \in \Sigma^\omega(f)$ be an $\omega$-singular point satisfying $\hat{\omega}_2(p)=0$ of a generalized timelike minimal surface $f$ with W-data $(g_1, g_2,\hat{\omega}_1du, \hat{\omega}_2dv)$. Then $f$ is $\mathcal{A}$-equivalent to the cuspidal beaks if and only if the following conditions hold:
\begin{equation}\label{eq:(BeaksLips}
\hat{\omega}_1\neq0, (\hat{\omega}_2)_v\neq 0, g_1g_2=1 \text{ and } (g_1)_u(g_2)_v\neq 0 \text{ at $p$.}
\end{equation}
Moreover, there is no cuspidal lips on any generalized timelike minimal surface in $\mathbb{L}^3$.
\end{proposition}
\begin{proof}
By Lemma \ref{lemma_front_condition}, $f$ is a front at $p$ if and only if $\hat{\omega}_1(p)\neq 0$ and $(g_2)_v(p)\neq 0$. Since $p$ is degenerate, $(\hat{\omega}_2)_v(p) =0$ or $g_1(p)g_2(p)=1$ also holds. Therefore, the function $\lambda$ in \eqref{eq:lambda} satisfies $\lambda_u=\lambda_v=\lambda_{uu}=0$ at $p$. Hence, the non-existence of cuspidal lips is proved by Fact \ref{ISTcriteria} and the relation $\det{(\Hess{\lambda})}(p)=-\lambda_{uv}(p)^2\leq 0$. Since 
\[
\lambda_{uv}=-\Lambda(g_1)_ug_2\hat{\omega}_1(\hat{\omega}_2)_v+\Lambda_u(1-g_1g_2)\hat{\omega}_1(\hat{\omega}_2)_v \text{ at $p$},
\]
it is sufficient to consider the case of $(\hat{\omega}_2)_v(p)\neq 0$ and $g_1(p)g_2(p)=1$.
In this case, $\lambda_{uv}=-\Lambda(g_1)_ug_2\hat{\omega}_1(\hat{\omega}_2)_v$ at $p$ and hence the condition $\det{(\Hess{\lambda})}(p) < 0$ is attained only when $(g_1)_u(p)\neq0$. Moreover, the null vector field $\eta=\partial_v$ satisfies
\[
\eta\eta\lambda=\lambda_{vv}=-\Lambda g_1(g_2)_v\hat{\omega}_1(\hat{\omega}_2)_v \neq 0 \text{ at $p$},
\]
which proves the desired result by Fact \ref{ISTcriteria}.
\end{proof}

\subsection{$D_4^\pm$ singularities}
$D_4$ singularities are typical rank zero singularities and consequently they are degenerate. In fact, as shown in Introduction, they appear on minimal surfaces in $\mathbb{E}^3$ as branch points.
First we give a necessary and sufficient condition for a generalized timelike minimal surface to be a front at a rank zero singular point. 
\begin{lemma}\label{lemma_rank_two}
Let $f=f(u,v)$ be a generalized timelike minimal surface with W-data $(g_1, g_2,\hat{\omega}_1du, \hat{\omega}_2dv)$. The following statements hold.
\begin{itemize}
\item[(i)] A point $p$ satisfies $df_p=\bm{0}$ if and only if $\hat{\omega}_1(p)=\hat{\omega}_2(p)=0$.
\item[(ii)] Let $p$ be a singular point with $df_p=\bm{0}$. Then, the surface $f$ is a front at  $p$ if and only if $g_1g_2(p)\neq 1$ and $(g_1)_u(g_2)_v(p)\neq 0$.
\end{itemize}
\end{lemma}

\begin{proof}Since 
$f_u=\hat{\omega}_1(-1-g_1^2, 1-g_1^2, 2g_1)/2$ and $f_v=\hat{\omega}_2(1+g_2^2, 1-g_2^2, -2g_2)/2$,
a point $p$ satisfies $df_p=\bm{0}$ if and only if $\hat{\omega}_i(p)=g_i\hat{\omega_i}(p)=g_i^2\hat{\omega}_i(p)=0$ for each $i=1,2$. It is equivalent to $\hat{\omega}_i(p)=0$ by the assumption that $g_i(p)\in\mathbb{R}$. Then we have proved the assertion (i). 

Next we prove the assertion (ii). Since $df_p= {\bm 0}$, $f$ is a front at $p$ if and only if $n_u(p)$ and $n_v(p)$ are linearly independent. By \eqref{eq:n_u} and \eqref{eq:n_v}, we have
\[
n_u\times n_v=(g_1)_u(g_2)_v(1-g_1g_2)\tilde{n},\quad \text{$\tilde{n}$ is a non-zero vector},
\]
which proves the assertion.
\end{proof}

Saji \cite{S2} gave useful criteria for $D_4$ singularities on fronts.
\begin{fact}[\cite{S2}]\label{saji_criteria}
Let $f\colon U\subset\mathbb{R}^2 \longrightarrow \mathbb{R}^3$ be a front and $p$ be a point in $U$. Then $f$ is $\mathcal{A}$-equivalent to the $D_4^+$ singularity (resp. $D_4^-$ singularity) if and only if $df_p=\bm{0}$ and $\det{(\Hess{\lambda)}}(p)<0$ (resp.\ $\det{(\Hess{\lambda})}(p)>0$) hold,
where $\det{(\Hess{\lambda})}$ means the determinant of the Hessian matrix of the function $\lambda$.
\end{fact}
Applying Fact \ref{saji_criteria}, we have the following. 

\begin{proposition}\label{D_4criterion}
Let $f=f(u,v)$ be a generalized timelike minimal surface with W-data $(g_1, g_2,\hat{\omega}_1du, \hat{\omega}_2dv)$. Then $f$ is $\mathcal{A}$-equivalent to the $D_4^+$ singularity if and only if
\begin{equation}\label{eq:D4criteria}
\hat{\omega}_1(p)=\hat{\omega}_2(p)=0, g_1g_2(p)\neq1, (g_1)_u(g_2)_v(p)\neq0  \text{ and } (\hat{\omega}_1)_u(\hat{\omega}_2)_v(p)\neq 0.
\end{equation}
Moreover there is no $D_4^-$ singularity on any generalized timelike minimal surface in $\mathbb{L}^3$.
\end{proposition}

\begin{proof}
By Lemma \ref{lemma_rank_two}, $f$ is a front at $p$ with $df_p=0$ if and only if 
\begin{equation}\label{eq:front_cond}
\hat{\omega}_1(p)=\hat{\omega}_2(p)=0, g_1g_2(p)\neq1 \text{ and } (g_1)_u(g_2)_v(p)\neq0.
\end{equation}
The signed area density function $\lambda$ in \eqref{eq:lambda} can be rewritten as
\[
\lambda=\tilde{\Lambda}\hat{\omega}_1\hat{\omega}_2,\quad \tilde{\Lambda}=-(1-g_1g_2)\sqrt{(1-g_1g_2)^2+2(g_1+g_2)^2}/2,
\] 
where $\tilde{\Lambda}$ satisfis $\tilde{\Lambda}(p)\neq0$ by the front condition \eqref{eq:front_cond}. Since
\[
\lambda_u=\tilde{\Lambda}_u\hat{\omega}_1\hat{\omega}_2+\tilde{\Lambda}(\hat{\omega}_1)_u\hat{\omega}_2,\quad \lambda_v=\tilde{\Lambda}_v\hat{\omega}_1\hat{\omega}_2+\tilde{\Lambda}\hat{\omega}_1(\hat{\omega}_2)_v,
\]
we have 
\[
\lambda_{uu}=\lambda_{vv}=0,\quad \lambda_{uv}=\tilde{\Lambda}(\hat{\omega}_1)_u(\hat{\omega}_2)_v\quad \text{at $p$.}
\]
Hence,
\begin{equation}\label{eq:dethess}
\det{(\Hess{\lambda)}}=-\lambda_{uv}^2=-\tilde{\Lambda}^2(\hat{\omega}_1)_u^2(\hat{\omega}_2)_v^2\quad \text{at $p$.}
\end{equation}

By \eqref{eq:front_cond}, \eqref{eq:dethess} and Fact \ref{saji_criteria}, we obtain the desired criterion. Moreover the right hand side of the equation \eqref{eq:dethess} is always nonpositive. This means that there is no $D_4^-$ singularity on any generalized timelike minimal surface by Fact \ref{saji_criteria}.
\end{proof} 

\begin{remark}
The {\it Hopf differentials} of $f$ are defined by 
\[ 
Qdu^2=\langle f_{uu}, n_L \rangle du^2,\quad Rdv^2=\langle f_{vv}, n_L \rangle dv^2
\]
where $n_L$ is a unit normal vector field of $f$ in $\mathbb{L}^3$. By a direct calculation, we can see that $Q$ and $R$ are written as $Q=\hat{\omega}_1(g_1)_u$ and $R=-\hat{\omega}_2(g_2)_v$ for an appropriate choice of $n_L$ and these relations make it possible to define $Q$ and $R$ across singular points. 
By using these $Q$ and $R$, the relation \eqref{eq:D4criteria} is equivalent to
\[
df_p=\bm{0}, p \in \Sigma^\omega(f)\setminus \Sigma^g(f) \text{ and } Q'R'(p)\neq0.
\]
\end{remark}

\begin{remark}
Ogata and Teramoto \cite{OT2} gave a criterion for $D_4^-$ singularities on surfaces  (resp.\ spacelike surfaces) in 3-dimensional Riemannian (resp.\ Lorentzian) spaceforms with constant mean curvature. They also show that there is no $D_4^+$ singularity on these surfaces, but the opposite phenomenon occurs for timelike surfaces in the sense of Proposition \ref{D_4criterion}.
\end{remark}

By proposition \ref{D_4criterion}, we also have the following invariance of $D_4^+$ singularities as well as $(2,5)$-cuspidal edges in Corollary \ref{cor:duality_(2,5)_cusp}.
\begin{corollary}\label{cor:duality}
Let $p$ be an $\omega$-singular point of a generalized timelike minimal surface $f$. Then, the following statements are equivalent.
\begin{itemize}
\item[(i)] $f$ is $\mathcal{A}$-equivalent to the $D_4^+$ singularity at $p$,
\item[(ii)] $f_{\theta}$ is $\mathcal{A}$-equivalent to the $D_4^+$ singularity at $p$ for arbitrary $\theta\in \mathbb{R}$ and 
\item[(iii)] $\hat{f}$ is $\mathcal{A}$-equivalent to the $D_4^+$ singularity at $p$.
\end{itemize}

\end{corollary}
We remark that the same invariance of $D_4^-$ singularities is also valid for generalized maximal surfaces in $\mathbb{L}^3$.
The following torus was found by Kim, Koh, Shin and Yang \cite[Example 2.7]{KKSY}.
\begin{example}
Let us consider the $2\pi$-periodic null curve 
\[
\gamma(t)=(1/2\sin{2t},1/2\sin{t}+1/6\sin{3t},1/2\cos{t}-1/6\cos{3t}),
\]
and the generalized timelike minimal torus $f(u,v)=(\gamma(u)+\gamma(v))/2$, $(u,v)\in [0,2\pi]^2.$
This torus is found by Kim Koh Shin and Yang \cite{KKSY} and we will show this surface has $D_4^+$ singularities.
The singular points at which $df=\bm{0}$ are
\begin{equation}\label{eq:KKSYsing}
(u,v)=((1+2k)\pi/4,(1+2l)\pi/4),\quad 0\leq k,l \leq 3.
\end{equation}
The W-data $(g_1, g_2,\hat{\omega}_1du, \hat{\omega}_2dv)$ of $f$ is determined as
\begin{align*}
g_1(u)=\frac{\sin{u}}{\cos{u}-1}&,\ g_2(v)=-\frac{\sin{v}}{\cos{v}+1},\\
\hat{\omega}_1(u)=\frac{\cos{2u}(\cos{u}-1)}{2}&,\ \hat{\omega}_2(v)=\frac{\cos{2v}(\cos{v}+1)}{2}.
\end{align*}
Here, we remark that the functions $g_1$ and $g_2$ have finite values at above singular points. We can see that W-data satisfies $(g_1)_u(g_2)_v(\hat{\omega}_1)_u(\hat{\omega}_2)_v\neq0$ at singular points \eqref{eq:KKSYsing}. Since $g_1((1+2k)\pi/4)g_2((1+2l)\pi/4)\neq 1$ for $k\neq l$, Proposition \ref{D_4criterion} shows that $f$ has these twelve $D_4^+$ singularities (see Figure \ref{Fig_D4+}).
 
Next we consider its conjugate surface. Since the generalized timelike minimal torus $f$ has the fold singular points (see \cite{FKKRSUYY_Okayama, KKSY} for details) 
\[
\Gamma=\{(u,u)\mid u \in [0,2\pi], u\neq \pi/4,3\pi/4,5\pi/4,7\pi/4\}
\]
with the folded symmetry $f(v,u)=f(u,v)$, its conjugate surface $\hat{f}$ has conelike singular points satisfying $\hat{f}(\Gamma)=\bm{0}$ in $\mathbb{L}^3$ with the point symmetry $\hat{f}(v,u)=-\hat{f}(u,v)$. By Corollary \ref{cor:duality}, the $D_4^+$ singularities on $f$ correspond to those of $\hat{f}$, see Figure \ref{Fig_ConjugateKKSY}. 

\begin{figure}[htb]
\vspace{-1.8cm}

\begin{center}
 \begin{tabular}{{c@{\hspace{-20mm}}c@{\hspace{-20mm}}c}}
\hspace{+3mm}   \resizebox{6.6cm}{!}{\includegraphics[clip,scale=0.33,bb=0 0 555 449]{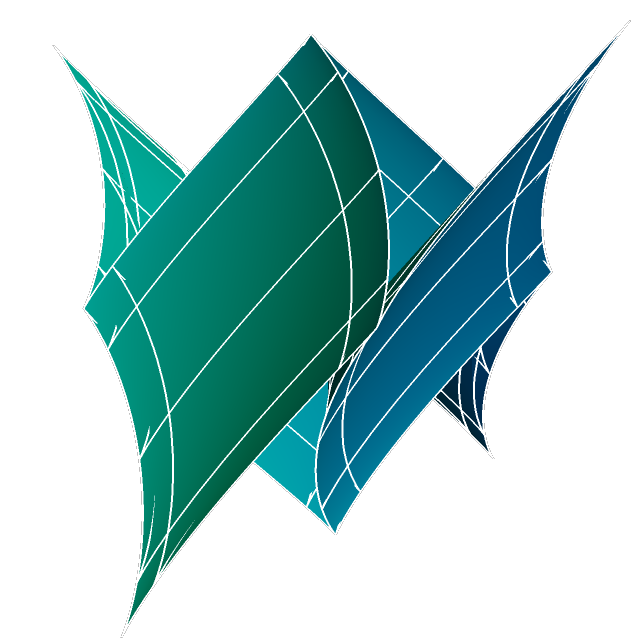}}&
  \resizebox{7.0cm}{!}{\includegraphics[clip,scale=0.35,bb=0 0 555 449]{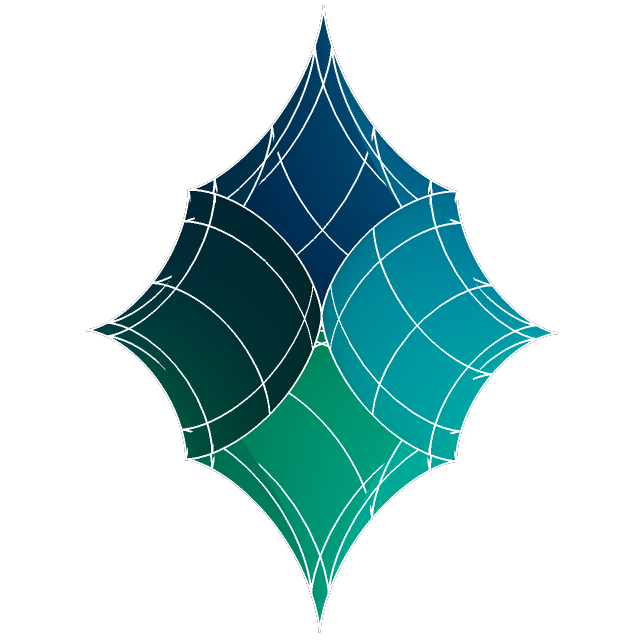}}&
  \resizebox{6cm}{!}{\includegraphics[clip,scale=0.38,bb=0 0 555 449]{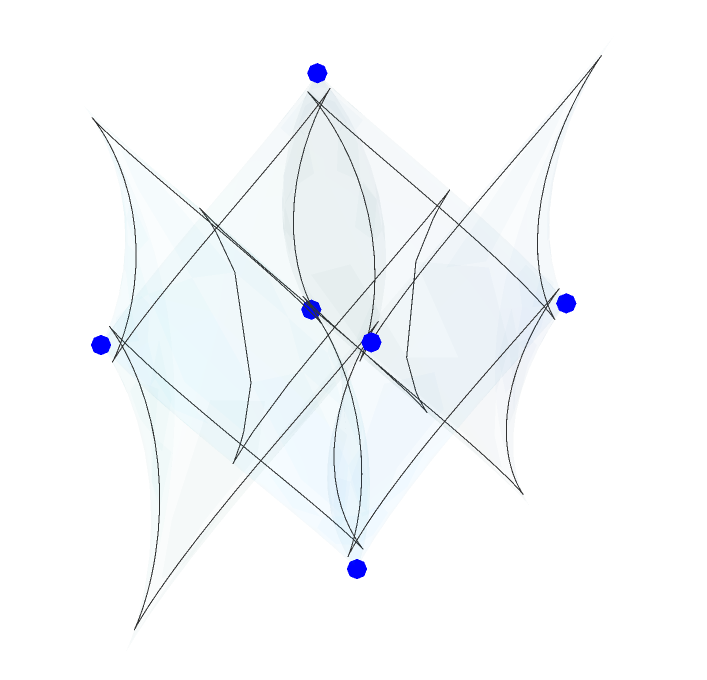}} \\
  {\hspace{-30mm}\footnotesize  The timelike minimal torus $f$.} &
  {\hspace{-30mm}\footnotesize  Another view.} &
  {\hspace{-20mm}\footnotesize  The images of $D_4^+$ singularities.}
 \end{tabular}
 \caption{The timelike minimal torus in \cite{KKSY} and the images of $D_4^+$ singularities on it (six dots in the right figure). Because of the folded symmetry, there is an overlap in the image of $D_4^+$ singularities.}
 \label{Fig_D4+}
\end{center}
\end{figure}

\begin{figure}[htb]
\vspace{-1.8cm}

\begin{center}
 \begin{tabular}{{c@{\hspace{-05mm}}c@{\hspace{-06mm}}c}}
 \hspace{+10mm}  \resizebox{5cm}{!}{\includegraphics[clip,scale=0.33,bb=0 0 555 449]{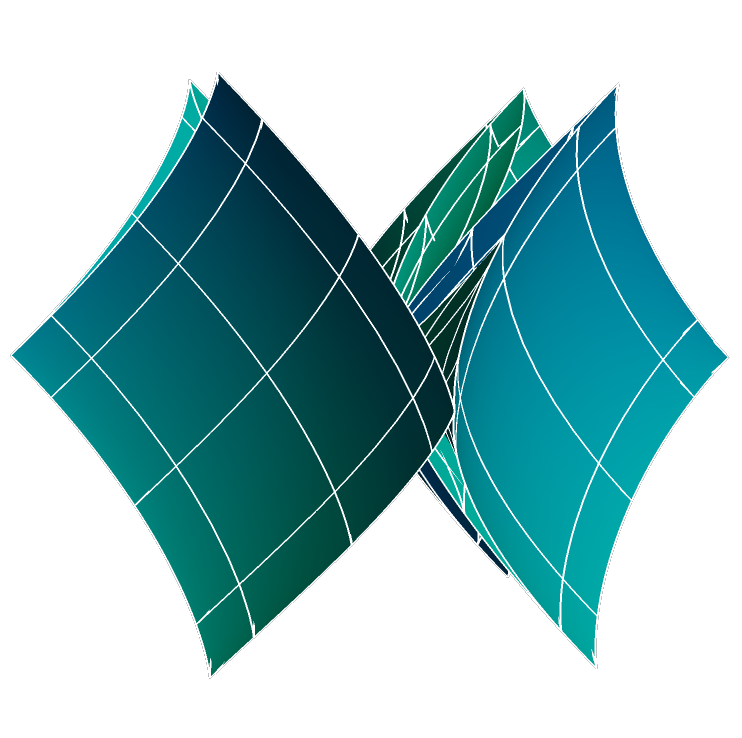}}&
  \resizebox{5cm}{!}{\includegraphics[clip,scale=0.35,bb=0 0 555 449]{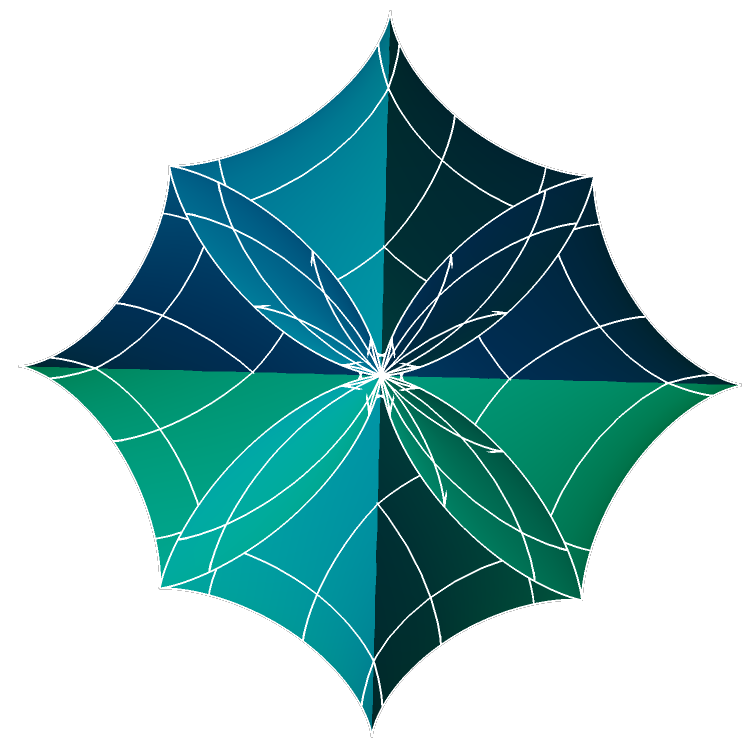}}&
  \resizebox{6cm}{!}{\includegraphics[clip,scale=0.38,bb=0 0 555 449]{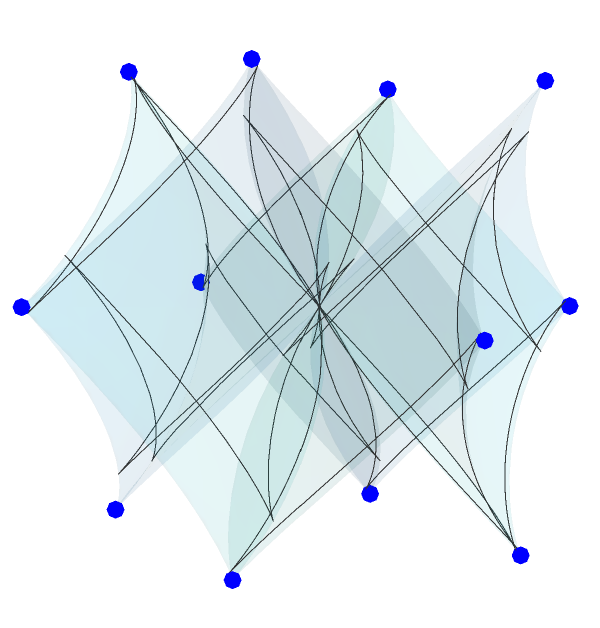}} \\
  {\hspace{-12mm}\footnotesize  The conjugate torus $\hat{f}$.} &
  {\hspace{-15mm}\footnotesize  Another view.} &
  {\hspace{-20mm}\footnotesize  The images of $D_4^+$ singularities.}
 \end{tabular}
 \caption{The conjugate surface $\hat{f}$ and the images of $D_4^+$ singularities on it (twelve dots in the right figure).}
\label{Fig_ConjugateKKSY}
\end{center}
\end{figure}

\end{example}

\appendix
\section{Proof of Lemma \ref{Lem:Lemma35_OT}}\label{sec:app}
In this section, we give a proof of Lemma \ref{Lem:Lemma35_OT}. Although the strategy of the proof is the same as in the case of maximal surfaces \cite[Lemma 35]{OT}, it should be noted that different assumptions are adopted, see also Remark \ref{remark:OT}.
\begin{proof}[Proof of Lemma \ref{Lem:Lemma35_OT}]
We can take a local coordinates $(U; \sigma,\tau)$ around $p$ such that $\xi = \partial_\sigma$ and $\eta=\partial_\tau$. 
Let us define the function $F(\sigma,\tau):=\langle f(\sigma,\tau), f_\sigma(\bm{0})\rangle_E$. Since $F_\sigma(\bm{0})=\langle f_\sigma(\bm{0}), f_\sigma(\bm{0})\rangle_E\neq 0$, there exists a function $\sigma=\sigma(\tau)$ such that $F(\sigma(\tau),\tau)=0$ and $\sigma(0)=0$. We show the curve $c(\tau)=(\sigma(\tau),\tau)$ satisfies the condition (iii) of Fact \ref{cS_criteria}.

Taking the derivative of $F\circ c$ with respect to $\tau$, we have 
\[
0=d(F\circ c)/d\tau = \langle f_\sigma(\sigma,\tau)\sigma'(\tau)+f_\tau(\sigma,\tau), f_\sigma(\bm{0})\rangle_E,
\]
where $\sigma'=d\sigma/d\tau$. Since $f_\tau=\eta f =\bm{0}$ along $\sigma$-axis, we have $\sigma'(0)=0$ and hence $c'(0)=\eta(0)$. 
Based on the relations $\sigma'(0)=0$ and $f_{\sigma\tau} =\bm{0}$ along $\sigma$-axis, we can compute 
\[
0=d^3(F\circ c)/d\tau^3 (0)= \langle f_\sigma(\bm{0})\sigma'''(0)+f_{\tau \tau \tau}(\bm{0}), f_\sigma(\bm{0})\rangle_E.
\]
By the assumption  $f_{\tau \tau \tau}(\bm{0})=\eta^3f(p)=\bm{0}$, we obtain the third derivative relation $\sigma'''(0)=0$. 
We note that, unlike the situation in \cite{OT}, $\sigma''(0)=0$ does not follow under the current assumption.

Next we check the condition (iii) of Fact \ref{cS_criteria} for $\hat{c}:=f\circ c$. Taking the derivatives of $\hat{c}$, we have
\begin{align*}
\hat{c}'&=f_\sigma\sigma'+f_\tau,\\
\hat{c}''&=f_{\sigma\sigma}(\sigma')^2+2f_{\sigma\tau}\sigma'+f_\sigma\sigma''+f_{\tau\tau},\\
\hat{c}'''&=f_{\sigma\sigma\sigma}(\sigma')^3+3f_{\sigma\sigma\tau}(\sigma')^2+3f_{\sigma\tau\tau}\sigma'+3f_{\sigma\sigma}\sigma'\sigma''+3f_{\sigma\tau}\sigma''+f_\sigma\sigma'''+f_{\tau\tau\tau}.
\end{align*}
So, by the relations $\sigma'(0)=\sigma'''(0)=0$ and $f_{\sigma\tau}(\bm{0})=f_{\tau\tau\tau}(\bm{0})=\eta^3f(p)=\bm{0}$, we obtain $\hat{c}'''(0)=\bm{0}$ and $\hat{c}''(0)=f_\sigma(\bm{0})\sigma''(0)+f_{\tau\tau}(\bm{0})$. By the assumption that $\xi f(p)$ and $\eta^2f(p)$ are linearly independent, $\hat{c}''(0)\neq\bm{0}$. Thus, we can take the constant $\ell$ in (iii) of Fact \ref{cS_criteria} as $\ell=0$. Moreover, the fifth derivative $c^{(5)}(0)$ is also computed as
\[
c^{(5)}(0)=10\xi \eta^3 f(p)+\sigma^{(5)}(0)\xi f(p)+ \eta^5f(p).
\]
Finally, we obtain
\[
B=\det{\left(\xi f(p), \hat{c}''(0), 3\hat{c}^{(5)}-10\ell \hat{c}^{(4)}\right)}=3\det{\left(\xi f(p),\eta^2f(p), \eta^5f(p)\right)},
\]
where we used the assumption that $\xi\eta^3f(p)$ and $\eta^2f(p)$ are linearly dependent.
\end{proof}

\textbf{Acknowledgement}
This work was partially supported by JSPS KAKENHI Grant Numbers 19K14527 and 23K12979. The author express his gratitude to Professors Atsufumi Honda, Miyuki Koiso, Yuta Ogata, Kentaro Saji, Masatomo Takahashi and Keisuke Teramoto for their valuable comments and suggestions. In particular, he would also like to express his sincere thanks to Professors Kentaro Saji and Masatomo Takahashi again for the opportunity to speak on topics related to this study at Geometry Seminar at Kobe University and series of lectures at Muroran University, and to the participants who attended those seminars.


\begin{thebibliography}{99}

\bibitem{Akamine}
 S.~Akamine, \emph{Behavior of the Gaussian curvature of timelike minimal surfaces with singularities}, Hokkaido Math.\ J.\ {\bf 48} (2019), 537--568.
 
 \bibitem{Akamine2}
 S.~Akamine, \emph{Isometric and anti-isometric classes of timelike minimal surfaces in Lorentz-Minkowski space}, preprint. arXiv:2305.04552.

\bibitem{ACO}
 S.~Akamine, J.~Cho, Y.~Ogata, \emph{Analysis of timelike Thomsen surfaces}, J.\ Geom.\ Anal. {\bf 30}(1) (2020), 731--761. 

\bibitem{AGV}
V.I.~Arnol'd, S.M.~Gusein-Zade and A.N.~Varchenko, \emph{Singularities of differentiable maps}, Vol.\ 1, Monogr.\ Math.\ {\bf 82}, Birkh\"{a}user Boston, Inc., Boston, MA, 1985.

\bibitem{BS}
D.~Brander and M.~Svensson, \emph{Timelike Constant Mean Curvature Surfaces with Singularities}, J. Geom. Anal. {\bf 24}, 1641--1672 (2014).

\bibitem{Calabi}
\textsc{E. Calabi}, \emph{Examples of the Bernstein problem for some nonlinear equations}, Proc.\ Symp.\ Pure Math.\ {\bf 15} (1970), 223--230.
 
 \bibitem{Catalan}
 E.~Catalan, \emph{M\'{e}moire sur les surfaces dont les rayons de courbures en chaque point, sont \'{e}gaux et les signes contraires}, C.R.Acad. Sci. Paris {\bf 41}, 1019--1023 (1855).
 
\bibitem{CY}
\textsc{S. Y. Cheng, S. T. Yau}, Maximal spacelike hypersurfaces in Lorentz-Minkowski space, Ann.\ of Math.\ (2) {\bf 104}, (1976), 407--419 .

\bibitem{FKKRSUYY_Okayama}
  S.~Fujimori, Y.W.~Kim, S.-E.~Koh, W.~Rossman, H.~Shin, M.~Umehara,
  K.~Yamada and S.-D.~Yang,
  \emph{Zero mean curvature surfaces in {L}orentz-{M}inkowski $3$-space
             and $2$-dimensional fluid mechanics}, 
  Math.\ J.\ Okayama Univ.\ {\bf 57} (2015), 173--200.

\bibitem{FT}
T.~Fukunaga and M.~Takahashi, \emph{Singularities of Translation Surfaces in the Euclidean $3$-Space}. Results Math {\bf 77}, 89 (2022).

\bibitem{FSUY}
S.~Fujimori, K.~Saji, M.~Umehara and K.~Yamada,
\emph{Singularities of maximal surfaces}, Math. Z. {\bf 259} (2008), 827--848.

\bibitem{Henneberg}
L.~Henneberg. \emph{\"{U}ber salche minimalfl\"{a}che, welche eine vorgeschriebene ebene curve sur geod\"{a}tishen line haben}, Doctoral Dissertation, Eidgen\"{o}ssisches Polythechikum, Zurich, 1875.

\bibitem{HKS}
A.~Honda, M.~Koiso and K.~Saji, 
\emph{Fold singularities on spacelike CMC surfaces in Lorentz-Minkowski space}, Hokkaido Math.\ J.\ {\bf 47} (2018), 245--267.

\bibitem{IS}
S.~Izumiya and K.~Saji,
\emph{The mandala of Legendrian dualities for pseudo-spheres in Lorentz-Minkowski space and ``flat'' spacelike surfaces}, J.\ Singul.\ {\bf 2} (2010), 92--127.

\bibitem{IST}
S.~Izumiya, K.~Saji and M.~Takahashi, \emph{Horospherical 
at surfaces in Hyperbolic 3-space}, J.\ Math.\ Soc.\ Japan {\bf 62} (2010), 789--849.

\bibitem{KKSY}
Y.~W.~Kim, S.-E.~Koh, H.~Shin and S.-D.~Yang,
\emph{Spacelike maximal surfaces, timelike minimal surfaces, and Bj\"orling representation formulae}, J.\ Korean Math.\ Soc.\ {\bf 48} (2011), 1083--1100.

\bibitem{KY}
Y.~W.~Kim and S.-D.~Yang, \emph{Prescribing singularities of maximal surfaces via a singular
Bj\"olring representation formula}, J. Geom. Phys. {\bf 57} (2007), no. 11, 2167--2177.

\bibitem{KN1}
S.V.~Klimenko, I.N.Nikitin, \emph{Singularities on world sheets of open relativistic strings}. Theor. Math. Phys. {\bf 114}, 299--312 (1998).

\bibitem{KN2}
S.V.~Klimenko, I.N.Nikitin, \emph{Exotic solutions in string theory}. Il Nuovo Cimento A (1971-1996) {\bf 111}, 1431--1455 (1998).

\bibitem{Kobayashi}
\textsc{O. Kobayashi}, Maximal surfaces in the 3-dimensional Minkowski space $\mathbb{L}^3$, Tokyo J.\ Math.\ {\bf 6} (1983), 297--309.

\bibitem{KRSUY}
M.~Kokubu, W.~Rossman, K.~Saji, M.~Umehara and K.~Yamada, 
\emph{Singularities of flat fronts in hyperbolic $3$-space}, Pacific J.\ Math.\ {\bf 221} (2005), 303--351.

\bibitem{McNertney}
L.~McNertney, 
\emph{One-parameter families of surfaces with constant curvature in Lorentz $3$-space}, Ph.D.\ thesis, Brown University (1980).

\bibitem{OT}
Y.~Ogata and K.~Teramoto,
\emph{Duality between cuspidal butterflies and cuspidal $S_1^{-}$ singularities on maxfaces}, Note Mat.~{\bf 38} (1) (2018), 115--130.
  
\bibitem{OT2}
Y.~Ogata and K.~Teramoto,
\emph{Constant mean curvature surfaces with $D_4$-singularities}, Adv.~Stud.~Pure Math.~{\bf 78} (2018), 345--363.

\bibitem{S1} 
K.~Saji, \emph{Criteria for cuspidal $S_k$ singularities and their applications}, J. G\"{o}kova Geom.~Topol.~GGT {\bf 4} (2010), 67--81.
 
\bibitem{S2}
 K.~Saji, \emph{Criteria for $D_4$ singularities of wave fronts}, Tohoku Math.\ J.\ {\bf 63} (2011), 137--147. 
 
\bibitem{SUY09}
K.~Saji, M.~Umehara and K.~Yamada, 
\emph{The geometry of fronts}, Ann. of Math. {\bf 169} (2009), 491--529.

\bibitem{T}
H.~Takahashi, 
\emph{Timelike minimal surfaces with singularities in three-dimensional spacetime} (in Japanese), Master thesis, Osaka University (2012).

\bibitem{UY}
M.~Umehara and K.~Yamada,
\emph{Maximal surfaces with singularities in Minkowski space}, Hokkaido Math. J. {\bf 35} (2006), 13--40.


\end{thebibliography}
 \end{document}